 \documentclass{amsart}
\usepackage{amsmath,amsthm,latexsym,amssymb,hyperref}
\usepackage{eucal}
\usepackage{mathrsfs}
\usepackage[all]{xy}
\numberwithin{equation}{section}
\newtheorem{thm}{Theorem}[section]
\newtheorem{lem}[thm]{Lemma}
\newtheorem{prop}[thm]{Proposition}
\newtheorem{cor}[thm]{Corollary}
\newtheorem*{moore}{Moore's Theorem}
\theoremstyle{definition}
\newtheorem{defn}[thm]{Definition}
\theoremstyle{remark}
\newtheorem{rmk}[thm]{Remark}

\newtheorem{ex}[thm]{Example}

\newcommand \Om{\Omega}

\newcommand \del{\partial}
\newcommand \vp{\varphi}
\newcommand \ve{\varepsilon}

\newcommand \si{s\sp{-1}}

\newdir{ >}{{}*!/-5pt/@{>}}

\newcommand\sn [1]{(-1)^{#1}}

\newcommand\tor{\operatorname{Tor}}

\renewcommand\H{\operatorname{H}}
\newcommand\op{\mathcal}
\newcommand\cat{\mathbf}
\renewcommand\Bar{\op B}
\newcommand \G {\mathsf G}
\newcommand\ot{\otimes}

\def\fatcoalg{(\op A,{\op F})\text{-}\cat {Coalg}}

\def\acirc{\diamond_{\op A}}

\begin{document}
\title {Multiplicative structure in equivariant cohomology}
\author {Kathryn Hess}
\address{MATHGEOM \\
    \'Ecole Polytechnique F\'ed\'erale de Lausanne \\
    CH-1015 Lausanne \\
    Switzerland}
    \email{kathryn.hess@epfl.ch}
\date \today

 \keywords {Homotopy orbits, strongly homotopy morphism, operad, circle action, simplicial group action} 
 \subjclass [2000] {Primary: 55R20 Secondary: 18D50, 18G30, 18G35, 18G55, 55P35, 55P40, 55U05, 55U10, 57T05}
 \begin{abstract} We introduce the notion of a strongly homotopy-comultiplicative resolution of a module coalgebra over a chain Hopf algebra, which we apply to proving a comultiplicative enrichment of a well-known theorem of Moore concerning the homology of quotient spaces of group actions.  The importance of our enriched version of Moore's theorem lies in its application to the construction of useful cochain algebra models for computing multiplicative structure in equivariant cohomology.
 
In the special cases of homotopy orbits of circle actions on spaces and of group actions on simplicial sets, we obtain small, explicit cochain algebra models that we describe in detail.
 \end{abstract}

 \maketitle
 
  \tableofcontents


\section*{Introduction}

Let $C_{*}X$ denote  the (singular or cubical) chain complex of a space $X$, which admits a natural coassociative and counital comultiplication $\delta _{X}$, given by the composite chain map
$$C_{*}X\xrightarrow{C_{*}\Delta} C_{*}(X\times X)\xrightarrow {AW} C_{*}X\otimes C_{*}X,$$
where $\Delta$ is the usual diagonal map and $AW$ is the natural Alexander-Whitney equivalence.   By the K\"unneth Theorem, if $H_{*}X$ is torsion free, then $\delta _{X}$ induces a comultiplication $H_{*}X\to H_{*}X\otimes H_{*}X$.  In general, $\delta _{X}$ induces a graded commutative multiplication $H^*X\otimes H^*X\to H^*X$, the cup product.

Let $E$ be the total space of a principal $G$-bundle, where $G$ is a connected topological group.  Let $Y$ be any $G$-space.  The multiplication map $\mu :G\times G\to G$ induces the structure of a chain algebra on $C_{*}G$, with multiplication map given by the composite
$$C_{*}G\otimes C_{*}G\xrightarrow {EZ} C_{*}(G\times G)\xrightarrow {C_{*}\mu }C_{*}G,$$
where $EZ$ is the natural Eilenberg-Zilber equivalence.  The action maps $E\times G\to E$ and $G\times Y\to Y$ similarly induce $C_{*}G$-module structures on $C_{*}E$ and on $C_{*}Y$.   

In \cite {moore} Moore proved the following fundamental result relating the $C_{*}G$-module structures on $C_{*}E$ and on $C_{*}Y$ to the homology of the quotient space $E\times_G Y$.

\begin{moore}\label{thm:moore} There is an isomorphism of graded $\mathbb Z$-modules
$$H_{*}(E\times_{G} Y)\cong \operatorname{Tor}_{*}^{C_{*}G}(C_{*}E,C_{*}Y).$$
\end{moore}

In this paper we explain how to enrich Moore's theorem, obtaining a comultiplicative isomorphism, by taking into account in a coherent manner the comultiplicative structure on $C_{*}G$, $C_{*}E$ and $C_{*}Y$, up to strong homotopy.  We then analyze in detail the special case $G=S^1$ and $E=ES^1$.   We also provide a small, explicit model for the homotopy orbits of a group action on a reduced simplicial set.

 We begin by recalling the operadic description of ``strongly homotopy'' structures from \cite {hps1} and \cite {hpst}, based on which we introduce a more highly structured notion of resolution, which we call \emph{DCSH-resolution} (Definition \ref{defn:resoln}).  We show that DCSH-resolutions lift through surjective quasi-isomorphisms (Theorem \ref{thm:exist-dcsh-res}), which is particularly useful for our purposes.    We then apply DCSH-resolutions to  proving an enriched version of Moore's theorem, showing that it is possible to calculate the algebra structure of $\H^*(E\times_{G} Y )$, given DCSH-resolutions of $C_{*}G$ and of  $C_{*}E$ as a $C_{*}G$-module (Theorem \ref{thm:enriched-moore}).

As an application of our enriched Moore's theorem, we consider the case of homotopy orbits of circle actions.  After proving the existence of a special family of primitive elements in the reduced cubical chains on $S^1$ (Definition \ref{defn:primfamily}) and studying its properties, we introduce a particularly useful DCSH-resolution of the cubical chains on $ES^1$ as a module over the cubical chains on $S^1$ (\ref{eqn:dcsh-resoln}).  From our enriched Moore's Theorem, we then obtain for any $S^1$-space $Y$ a cochain algebra the cohomology of which is isomorphic to the graded algebra $\H^*Y_{hS^1}$  (\ref{eqn:cochain-model-circle}).   The Batalin-Vilkovisky structure on $\H^*Y$ appears in this cochain algebra, as one summand of the differential.

In the last section we apply our enriched Moore's theorem to constructing analogous small, explicit chain coalgebra models for the homotopy orbits of a group action on a reduced simplicial set (Theorem \ref{thm:new-moore-gk} ).

Throughout this paper, we elaborate upon and improve certain results from \cite {hess}.  In particular, the operadic perspective on computing the algebra structure of equivariant cohomology is new.

\subsection*{Related work}
Neisendorfer strengthened Moore's Theorem in \cite[Theorem 12.12.1]{neisendorfer}, proving that if $E$ is a free right $G$-space and $Y$ is any $G$-space such that $\H_{*}(E\times_{G}Y)$ is $\Bbbk$-projective, then 
$$\H_{*}(E\times_{G}Y)\cong \tor _{*}^{C_{*}G}(C_{*}E, C_{*}Y)$$
as coalgebras, where the coalgebra structure on the right arises from that of the acyclic bar construction on $C_{*}G$ and that of $C_{*}Y$ (cf.~Example \ref{ex:canonical-res}). In this article, we show that this canonical bar resolution can be replaced by any DCSH-resolution when calculating comultiplicative structure, which has the advantage of enabling us to use particularly small resolutions.

In \cite[Theorem 5.1]{f四ix-halperin-thomas}, F\'elix, Halperin and Thomas proved a result similar to that of Neisendorfer for \emph{$G$-fibrations}, i.e., for fibrations $p:E\to B$ such that $E$ admits a fiberwise right $G$-action inducing weak equivalences $G \to p^{-1}\big(p(e)\big): a \mapsto e\cdot a$ for all $e\in E$.  They showed that there was a natural quasi-isomorphism of chain coalgebras
$$C_{*}E\ot_{t_{\Bar}} \Bar C_{*}G\xrightarrow \simeq C_{*}B$$
for every such $G$-fibration, where $\Bar$ is the (reduced) bar construction, and $t_{\Bar}$ is the couniversal twisting cochain (Example \ref{ex:couniversal}).

\subsection*{Notation and conventions}
\begin{itemize}
\item Given objects $A$ and $B$ of a category $\cat C$, we let $\cat C(A,B)$ denote the set of morphisms with source $A$ and target $B$. 
\item When used in the name of a morphism, $1$ or $1_{X}$ indicates the identity morphism on an object $X$.
\item Throughout this paper we are working over a commutative ring $\Bbbk$.  We denote the category of $\mathbb Z$-graded $\Bbbk$-modules by $\cat{grMod}_\Bbbk$ and the category of unbounded chain complexes over $\Bbbk$ by $\cat{Ch}_\Bbbk$.   
\item The degree of an element $v$ of a graded module $V$ is denoted either $|v|$ or simply $v$, when used as an exponent, and no confusion can arise.
\item Throughout this article we apply the Koszul sign convention for commuting elements  of a graded module or for commuting a morphism of graded modules past an element of the source module.  For example,  if $V$ and $W$ are graded algebras and $v\otimes w, v'\otimes w'\in V\otimes W$, then $$(v\otimes w)\cdot (v'\otimes w')=(-1)^{|w|\cdot |v'|}vv'\otimes ww'.$$ Futhermore, if $f:V\to V'$ and $g:W\to W'$ are morphisms of graded modules, then for all $v\otimes w\in V\otimes W$, 
$$(f\otimes g)(v\otimes w)=(-1)^{|g|\cdot |v|} f(v)\otimes g(w).$$ 
\item Dualization with respect to $\Bbbk$ is indicated  by a $\sharp$ as superscript.

\item A graded module $V$ is \emph {connected} if $V_{k}=0$ for all $k<0$ and $V_{0}=\Bbbk$.  We write $V_{+}$ for $V_{>0}$.
\item The \emph {suspension} endofunctor $s$ on the category of graded modules is defined on objects $V=\bigoplus _{i\in \mathbb Z} V_ i$ by
$(sV)_ i \cong V_ {i-1}$.  Given a homogeneous element $v$ in
$V$, we write $sv$ for the corresponding element of $sV$. The suspension $s$ admits an obvious inverse, which we denote $\si$.
\item Given chain complexes $(V,d)$ and $(W,d)$, the notation
$f:(V,d)\xrightarrow{\simeq}(W,d)$ indicates that $f$ induces an isomorphism in homology. 
In this case we refer to $f$ as a \emph { quasi-isomorphism}.
\item Let $T$ denote the endofunctor on the category of free graded $\Bbbk$-modules given by
$$TV=\oplus _{n\geq 0}V^{\otimes n},$$
where $V^{\otimes 0}:=\Bbbk$.  A pure tensor element of the summand $V^{\otimes n}$ of $TV$ is denoted $v_{1}|\cdots |v_{n}$, where $v_{i}\in V$ for all $i$. 
\item If $K$ is a simplicial set, then $C_{*}K$ denotes its normalized chain complex with coefficients in $\Bbbk$. If $X$ is a topological space, then  $C_{*}X$ means either its cubical or its singular chain complex with coefficients in $\Bbbk$.  We use the notation $CU_{*}X$ to specify the cubical chain complex.  The homology functor $\H_{*}$ for spaces or simplicial sets is always taken with coefficients in $\Bbbk$.
\end{itemize}

\section{The category $\mathbf{DCSH}$}\label{subsec:dcsh}

In this section we recall from \cite{gugenheim-munkholm} the definition of the category $\cat{DCSH}$, in which the objects are chain coalgebras and the morphisms are strongly homotopy-co\-mul\-ti\-pli\-ca\-tive maps.  We then remind the reader of the operadic description of $\cat {DCSH}$, as developed in \cite{hps1} and \cite{hpst}.  One advantage of this operadic description is that it enables us to see easily that $\cat {DCSH}$ admits a monoidal structure, studied in detail in \cite{hess-levi} and summarized briefly here.  In general, the operadic language we use simplifies both the presentation and the manipulation of the morphisms with which we work here.

The category $\mathbf {DCSH}$, first defined by Gugenheim and Munkholm in \cite {gugenheim-munkholm}, has as objects connected, coaugmented  chain coalgebras, while a morphism from $C$ to $C'$ is a map of chain algebras $\Omega C\to \Omega C'$, where $\Om$ denotes the (reduced) cobar construction (cf.~Appendix \ref{sec:twistingcochains}).  

In a slight abuse of terminology, we say that a chain map between chain coalgebras $f:C\to C'$ is a \emph{DCSH-map} if there is a morphism  in $\mathbf {DCSH}(C,C')$ of which $f$ is the linear part. In other words, there is a map of chain algebras $\varphi:\Om C\to \Om C'$ such that 
$$\varphi|_{s^{-1}C_+}=\si f s +\text{higher-order terms}.$$

\begin{rmk}\label{rmk:unravel-dcsh} Let  $\varphi: \Om C\to \Om C'$ be a chain algebra map, where $(C,d_{C},\Delta)$ and $(C',d_{C'}, \Delta')$ are connected, coaugmented chain coalgebras. The algebra map $\varphi$ is determined by its values on $s^{-1} C_{+}$, which generates the free associative algebra underlying $\Om C$.  Let $\varphi_{k}$ denote the following composite
$$C_{+}\xrightarrow {s^{-1}} s^{-1}C_{+}\hookrightarrow \Om C\xrightarrow\varphi \Om C'\xrightarrow{\text{proj.}}(s^{-1}C')^{\otimes k}\xrightarrow {s^{\otimes k}} (C')^{\otimes k}.$$

Unraveling the definition of the differential in the cobar construction, we see that specifying $\varphi$ is equivalent to giving a family of $\Bbbk$-linear maps
$$\{ \varphi_{k}:C\to (C')^{\otimes k}\mid \deg \varphi_{k}=k-1, k\geq 1\}$$
such that
\begin{multline*}
d_{(C')^{\otimes k}}\varphi_{k}+(-1)^k\varphi _{k}d_{C}\\=\sum _{i=1}^{k-1} (\varphi _{i}\otimes \varphi _{k-i})\overline\Delta _{C}-\sum _{i=0}^{k-2}(-1)^{i} (1^{\otimes i}\otimes \overline\Delta _{C'}\otimes 1^{k-i-2})\varphi_{k-1}
\end{multline*}
for all $k$. It follows that a chain map $f:C\to C'$ is a DCSH-map if there is such a family $\{\varphi_{k}\mid k\geq 1\}$, where $\varphi_{1}=f$.
\end{rmk}

The category $\mathbf {DCSH}$ plays an important role in topology.  In particular, as established in \cite {gugenheim-munkholm}, for any reduced simplicial set $K$, the usual comultiplication on $C_{*}K$  is a DCSH-map.  In \cite {hps1} the authors provided a purely operadic description of $\mathbf {DCSH}$.  Before recalling this description, we briefly explain the framework in which it is constructed.  We refer the reader to Appendix A of \cite {hps1} for further details. 

Let $\mathbf M$ denote either $\cat{grMod}_{\Bbbk}$ or $\cat{Ch}_{\Bbbk}$, and let $\mathbf M^\Sigma $denote the category of symmetric sequences of objects in $\cat M$.  An object $\mathcal X$ of $\mathbf M^\Sigma$ is  a family $\{\mathcal X(n)\in \mathbf M\mid n\geq 0\}$ of objects in $\mathbf M$ such that $\mathcal X(n)$ admits a right action of the symmetric group $\Sigma_n$, for all $n$.  The object $\mathcal X(n)$ is called the \emph{ $n^{\text{th}}$-level} of the symmetric sequence $\mathcal X$.  Given a morphism of symmetric sequences $\varphi:\mathcal X\to \mathcal Y$, we let $\varphi(n):\mathcal X(n)\to \mathcal Y(n)$ denote its restriction to level $n$. 

There is a faithful functor $\xymatrix@1{\mathcal T: \mathbf M\ar [r]&\mathbf M^\Sigma}$ where, for all $n$, $\mathcal T(A)(n)=A^{\otimes n}$, where $\Sigma _n$ acts by permuting the tensor factors.  The functor $\mathcal T$ is strong monoidal, with respect to the \emph{ level monoidal structure} $(\mathbf M^\Sigma, \otimes, \op C)$, where $(\op X\otimes \op Y)(n)=\op X(n)\otimes\op Y(n)$, endowed with the diagonal action of $\Sigma_n$, and $\op C(n)=\Bbbk$, endowed with the trivial $\Sigma _n$-action.

The category $\mathbf M^\Sigma$ also admits a nonsymmetric, right-closed monoidal structure $(\mathbf M^\Sigma,\diamond , \mathcal J)$, where $\diamond$ is the \emph{ composition product} of symmetric sequences, and $\mathcal J(1)=\Bbbk$ and $\mathcal J(n)=0$ otherwise.  Given symmetric sequences $\op X$ and $\op Y$,
$$(\op X\diamond\op Y)(n)=\coprod_{\substack{k\geq 1\\ \vec n\in I_{k,n}}}\op X(k)\underset {\Sigma _k} \otimes \bigl(\op Y(n_1)\otimes \cdots\otimes \op Y(n_k)\bigr)\underset {\Sigma _{\vec n}}\otimes \Bbbk[\Sigma _n],$$
where 
$$I_{k,n}=\{\vec n=(n_1,...,n_k)\in \mathbb  N^k\mid \sum _j n_j=n, n_{j}\geq 0 \,\forall j\}$$
 and 
 $$\Sigma_{\vec n}=\Sigma _{n_1}\times \cdots\times \Sigma _{n_k},$$ 
 seen as a subgroup of $\Sigma _n$.
For any objects $\mathcal X, \mathcal X', \mathcal Y, \mathcal Y'$ in $\mathbf M^\Sigma$, there is an obvious, natural intertwining map
\begin{equation}\label{eqn:intertwine}
\xymatrix{\iota: (\mathcal X\otimes \mathcal X')\diamond (\mathcal Y\otimes \mathcal Y')\ar [r]&(\mathcal X\diamond \mathcal Y)\otimes (\mathcal X'\diamond\mathcal Y')}.
\end{equation}

An \emph{ operad} in $\mathbf M$ is a monoid with respect to the composition product.  The \emph{ associative operad} $\mathcal A$ is given by $\mathcal A(n)=\Bbbk[\Sigma _n]$ for all $n$, endowed with the obvious monoidal structure, induced by permutation of blocks.  

We work throughout this paper with modules over operads.  Given operads $\op P$ and $\op Q$, we consider symmetric sequences $\op X$ admitting left or right actions of an operad $\op P$ or compatible left $\op P$-actions and right $\op Q$-actions, with respect to the composition product.  The categories of left $\op P$-modules, of right $\op P$-modules and of $(\op P,\op Q)$-bimodules are denoted  ${}_{\op 
P}\cat {Mod}$, $\cat {Mod}_{\op P}$, and ${}_{\op P}\cat {Mod}_{\op Q}$, respectively.

\begin{ex}  For all objects $A$ in $\cat M$, the tensor symmetric sequence $\op T(A)$ admits a natural, obvious left $\op A$-action.
\end{ex}

Let $\mathcal P$ denote any operad in $\mathbf M$.  A \emph{ $\mathcal P$-coalgebra} consists of an object $C$ in $\mathbf M$. together with an appropriately equivariant and associative family $$\{\xymatrix@1{C\otimes \mathcal P (n) \ar [r]&C^{\otimes n}\mid n\geq 0}\}$$ of morphisms in $\mathbf M$.  As observed in \cite{hps1}, the functor $\mathcal T$ restricts to a full and faithful functor $$\xymatrix@1{\mathcal T: \op P\text{-}\cat {Coalg}\ar [r]&{}_{\op A}\mathbf{Mod}_{\mathcal P}}$$ from the category of $\mathcal P$-coalgebras to the category of symmetric sequences with compatible left $\op A$-action and right $\mathcal P$-action.

Given a right $\op P$-module $\op M$ and a left $\op P$-module $\op N$, we can define their \emph{ composition product over $\op P$}, denoted $\op M\underset {\op P} \diamond \op N$, to be the coequalizer of the two obvious  maps $\op M\diamond \op P\diamond \op N\to \op M\diamond \op N$ induced by the right and left actions of $\op P$.

In \cite {hps1} the authors constructed an $\mathcal A$-bimodule $\mathcal F$, called the \emph {Alexander-Whitney co-ring}, which they applied to providing an operadic description of $\mathbf{DCSH}$.  The bimodule $\op F$ admits a coassociative comultiplication $\psi_{\op F}: \op F\to \op F\acirc \op F$ in the category of $\op A$-bimodules (with respect to the monoidal product $\acirc$), giving rise to its status as a co-ring.   There is also a coassociative, level comultiplication $\Delta_{\op F}:\op F\to \op F\otimes \op F$ that is compatible with its composition comultiplication  and that plays an important role in development of monoidal structure in $\cat {DCSH}$, which we exploit in section \ref{subsec:resolution}.

\begin{rmk}\label{rmk:filtered-structure}The symmetric sequence of graded $\Bbbk$-modules underlying $\mathcal F$ is $\op A\diamond \op S\diamond \op A$, where, for all $n\geq 1$, $\op S(n)= \Bbbk[\Sigma _n]\cdot z_{n-1}$, the free $\Bbbk[\Sigma _n]$-module on a generator of degree $n-1$, and $\op S(0)=0$.   We refer the reader to pages 853 and 854 in  \cite{hpst} for the explicit formulas for the differential $\del_{\op F}:\op F\to \op F$, the composition comultiplication $\psi _{\op F}$  and the level comultiplication $\Delta_{\op F}$.  We remark that $\op F$ admits a natural filtration with respect to which both $\psi_{\op F}$ and $\Delta _{\op F}$ are filtration-preserving, while $\del_{\op F}$ is filtration-decreasing.
\end{rmk}

Let $(\mathcal A,{\op F})\text{-}\mathbf {Coalg}$ denote the category of which the objects are $\mathcal A$-coalgebras (i.e., coassociative and counital chain coalgebras)  and where the morphisms are defined by 
$$\fatcoalg(C,C'):=_{\op A}\negthinspace\negthinspace\mathbf{Mod}_{\mathcal A}\bigl(\mathcal T(C)\acirc\mathcal F, \mathcal T(C')\bigr).$$
Composition in $(\mathcal A,{\op F})\text{-}\mathbf {Coalg}$ is defined in terms of $\psi_{\op F}$. If $\theta\in (\mathcal A,{\op F})\text{-}\mathbf {Coalg}(C, C')$ and $\theta'\in (\mathcal A,{\op F})\text{-}\mathbf {Coalg}(C', C'')$, then their composite $\theta'\theta\in(\mathcal A,{\op F})\text{-}\mathbf {Coalg}(C, C'')$ is given by composing the following sequence of (strict) morphisms of right $\op A$-modules.
$$\xymatrix@1{\op T(C)\acirc \op F\ar [rr]^(0.45){1_{\op T(C)}\acirc \psi_{\op F}}&&\op T(C)\acirc \op F \acirc \op F\ar [rr]^(0.55){\theta \acirc 1_{\op F}}&&\op T(C')\acirc \op F\ar [r]^{\theta'}&\op T(C'')}.$$
We call $(\mathcal A,{\op F})\text{-}\mathbf {Coalg}$ the category of $\op A$-coalgebras and of $\op F$-parametrized morphisms.  This is the promised operadic description of $\cat {DCSH}$, as the coKleisli category associated to the comonad  $-\acirc \op F $.

\begin{thm}\label{thm:hps} \cite {hps1} There is an isomorphism of categories $$\fatcoalg\cong\mathbf {DCSH}.$$\end{thm}

Define a bifunctor 
$$\wedge :\fatcoalg\times \fatcoalg \to  \fatcoalg$$
 on objects by $C\wedge C':= C\otimes C'$, the usual tensor product of chain coalgebras.  Given $ \theta\in \fatcoalg (C, D)$ and $\theta'\in \fatcoalg (C',D')$, we define $\theta \wedge \theta'$ to be the composite of (strict) $\op A$-bimodule maps
\begin{equation}\label{eqn:monoidal-fatcoalg}
\xymatrix{\op T(C\wedge C')\acirc \op F\ar [r]^(0.4)\cong\ar [rrddd]_{\theta \wedge \theta '}&\bigl(\op T(C)\otimes \op T(C')\bigr)\acirc \op F\ar [r]^(0.45){1\acirc \Delta _{\op F}}&\bigl(\op T(C)\otimes \op T(C')\bigr)\acirc (\op F\otimes \op F)\ar [d]^{\iota}\\
&&\bigl(\op T(C)\acirc \op F\bigr )\otimes\bigl(\op T(C')\acirc \op F\bigr )\ar [d]^{\theta \otimes \theta '}\\
&&\op T(D)\otimes \op T(D')\ar [d]^\cong\\
&&\op T(D\wedge D')}
\end{equation}
where $\iota $ is induced by the intertwining map (\ref{eqn:intertwine}).
It is straightforward to show that $\wedge$ endows $\fatcoalg$ with the structure of a monoidal category.  The finer details of this monoidal structure are studied in \cite[Section 2]{hess-levi}.

\begin{rmk}  It is clear that the category of chain coalgebras with its usual monoidal structure is a sub monoidal category of $\fatcoalg$, since every strict coalgebra map is strongly homotopy-comultiplicative, with trivial higher homotopies.
\end{rmk}

\section{Monoids and modules in $\fatcoalg$}\label{subsec:monmod}

It is clear that any chain Hopf algebra is a monoid with respect to the evident monoidal structure on $\cat{DCSH}$, since its multiplication map is a map of chain coalgebras and therefore a DCSH-map with trivial higher homotopies.  Relaxing the definition of a morphism of Hopf algebras, we introduce in this section the notion of \emph{DCSH-multiplicative maps} between chain Hopf algebras, which are strictly multiplicative but comultiplicative only up to strong homotopy, where the higher homotopies must be appropriately compatible with the multiplicative structure.  The definition of DCSH-multiplicative maps can be succintly stated in terms of the operadic description of the category $\cat {DCSH}$.

Relaxing analogously the notion of morphism of module coalgebras, we then define \emph{DCSH-module maps} between module coalgebras, which respect the multiplicative structure strictly but the comultiplicative structure only up to strong homotopy.  We prove in particular that tensoring DCSH-module maps over a DCSH-multiplicative map gives rise to a strongly homotopy-comultiplicative map.

\begin{defn}  Suppose that $H$ and $K$ are chain Hopf algebras.  A chain map $\xymatrix@1{\theta:H\ar [r]&\;K}$ is a \emph{ multiplicative DCSH-map} if there is a map of $\mathcal A$-bimodules
$$\xymatrix{{\widehat \theta}: \mathcal T(H)\acirc \op F\ar [r]&\mathcal T(K)}$$
such that $\widehat \theta(1)(x\otimes z_0)=\theta (x)$ for all $x\in H$ and such that 
$$\xymatrix{\op T(H\wedge H)\acirc\op F\ar [rr]^{\op T(\mu _H)\acirc 1}\ar[d]^{\widehat \theta\wedge\widehat\theta}&&\op T(H)\acirc\op F\ar [d]^{\widehat\theta}\\
\op T(K\wedge K)\ar [rr]^{\mathcal T(\mu _K)}&&\op T(K)}$$
commutes, where $\mu _H$ and $\mu _K$ are the multiplication maps of $H$ and $K$, which are maps of coalgebras.\end{defn}

In other words, a chain map $\theta$ between chain Hopf algebras is a multiplicative DCSH-map if it is the level-one part of a monoid morphism in $\fatcoalg$.

\begin{rmk}\label{rmk:unravel-mult-dcsh}  Just as we unraveled the definition of DCSH-maps in Remark \ref{rmk:unravel-dcsh}, we can provide a more explicit, though less compact, definition of multiplicative DCSH-maps as follows.  Let $H$ and $K$ be chain Hopf algebras.  A DCSH-map $\theta:H\to K$ with corresponding family of linear maps $\{\theta_{k}:H\to K^{\otimes k}\mid k\geq 1\}$ is multiplicative if 
$$\theta_{n}(ab)=\sum_{\substack{1\leq k\leq n\\ \vec\imath\in I_{k,n}}}\pm (\Delta_{H'}^{(i_{1})}\otimes \cdots \otimes\Delta_{H'}^{(i_{k})} )\theta_{k}(a)\cdot (\theta_{i_{1}}\otimes \cdots \otimes \theta_{i_{k}})\Delta _{H}^{(k)}(b),$$
for all $a,b\in H$ and for all $n\geq 1$, where $\cdot$ denotes multiplication in $(H')^{\otimes n}$ and where the signs are determined by the Koszul rule.  In particular, since $\theta=\theta_{1}$,
$$\theta(xy)=\theta (x)\theta(y)$$
for all $x,y\in H$, i.e., a multiplicative DCSH-map is, in particular, an algebra map.
\end{rmk}

The category of modules over a Hopf algebra $H$ admits a monoidal structure defined in terms of the comultiplication $\Delta$ on $H$: given two right $H$-modules $M$ and $M'$ with action maps $\rho$ and $\rho'$, their tensor product $M\otimes M'$ admits a right $H$-action given by the composite
$$(M\otimes M')\otimes H\xrightarrow{1\otimes \Delta} (M\otimes M')\otimes (H\otimes H)\xrightarrow {\cong} (M\otimes H)\otimes (M'\otimes H)\xrightarrow {\rho\otimes \rho'} M\otimes M'.$$
An $H$-module coalgebra is a comonoid in the category of $H$-modules, with respect to this monoidal structure.  We can also formulate the definition as follows.

\begin{defn}Let $H$ be a chain Hopf algebra.  A chain complex $M$ is a (right) \emph{ $H$-module coalgebra} if $M$ admits a chain coalgebra structure and a (right) $A$-module structure such that the $H$-action map $\xymatrix@1{\rho: M\otimes H\ar [r]& M}$ is a coalgebra map.\end{defn}

Embedding the category of chain coalgebras in $\fatcoalg$ as usual, we see that any module coalgebra over a Hopf algebra $H$ is a module over $H$, seen as a monoid in $\fatcoalg$.  We are therefore again naturally led to consider a weakened notion of morphism, this time between module coalgebras.

\begin{defn}\label{defn:dcsh-module}Let $\xymatrix@1{\theta:H\ar [r]&K}$ be a multiplicative DCSH-map.  Let $M$ and $N$ be a right $H$-module coalgebra and a right $K$-module coalgebra, respectively, where $\rho _M$ and $\rho_N$ are the module structure maps.  A chain map $\xymatrix@1{\vp :M\ar [r]&N}$ is a \emph{ DCSH-module map} with respect to $\theta$ if there is a map of $\op A$-bimodules
$$\xymatrix{{\widehat\vp}:\op T(M)\acirc \op F\ar [r]&\op T(N)}$$
such that $\widehat\vp (1)(y\otimes z_0)=\vp (y)$ for all $y\in M$ and such that   
$$\xymatrix{\op T(M\otimes H)\acirc\op F\ar [rr]^{\op T(\rho _M)\acirc 1}\ar [d]^{\widehat\vp\wedge\widehat\theta}&&\op T(M)\acirc\op F\ar [d]^{\widehat\vp}\\
\op T(N\otimes K)\ar [rr]^{\mathcal T(\rho _N)}&&\op T(N)}$$
commutes.
\end{defn}

\begin{rmk}\label{rmk:unravel-module-dcsh} In the spirit of Remarks \ref{rmk:unravel-dcsh} and \ref{rmk:unravel-mult-dcsh}, we now give an description of DCSH-module maps in terms of elements.  If $\theta:H\to K$ is a multiplicative DCSH-map with associated family $\{\theta_{k}:H\to K^{\otimes k}\mid k\geq 1\}$ and $M$ and $N$ are a right $H$-module coalgebra and a right $K$-module coalgebra, respectively, then a chain map $\varphi: M\to N$ is a DCSH-module map with respect to $\theta$ if it is a DCSH-map with associated family $\{\varphi_{k}:M\to N^{\otimes k}\mid k\geq 1\}$ such that
$$\varphi_{n}(x\cdot a)=\sum_{\substack{1\leq k\leq n\\ \vec\imath\in I_{k,n}}}\pm (\Delta_{N}^{(i_{1})}\otimes \cdots \otimes\Delta_{N}^{(i_{k})} )\varphi_{k}(x)\cdot (\theta_{i_{1}}\otimes \cdots \otimes \theta_{i_{k}})\Delta _{H}^{(k)}(a),$$
for all $x\in M$ and $a\in H$ and for all $n\geq 1$, where $\cdot$ denotes either the right action of $H$ on $M$ or the induced right action of $K^{\otimes n}$ on $N^{\otimes n}$ and where the signs are determined by the Koszul rule.  In particular,
$$\varphi (x\cdot a)=\varphi_{1}(x\cdot a)=\varphi _{1}(x)\cdot\theta_{1}(a)=\varphi (x)\cdot \theta (a),$$
i.e., $\varphi$ is itself a strict morphism of $H$-modules.
\end{rmk}

The definition of DCSH-module maps of left module coalgebras can be deduced easily from the definition above.

Let $H$ be a chain Hopf algebra. Suppose that $M$ and $M'$  are right and left $H$-module coalgebras, with structure maps $\rho_M$ and $\lambda _{M'}$, respectively. 
Consider the following coequalizer of chain complexes.
$$\xymatrix{M\otimes H\otimes M'\ar @<1ex>[rr]^{1\otimes \lambda_{M'}}\ar@<-1ex>[rr]_{\rho_M\otimes 1}&&M\otimes M'\ar [r]^{\pi}&M\otimes_{H} M'}$$
Since $\rho_M \otimes 1$ and $1\otimes \lambda_{M'}$ are both maps of coalgebras, $M\otimes_{H} M'$ admits a coalgebra structure such that the quotient map
$$\pi : M\otimes M'\to  M\otimes _{H}M'$$
is a coalgebra map.

In the next theorem we see that tensoring two DCSH-module maps over a DCSH-multiplicative map preserves strong homotopy-comultiplicativity.

\begin{thm}\label{thm:tensor-dcsh} Let $\xymatrix@1{\theta:H\ar [r]&K}$ be a multiplicative DCSH-map. Let $M$ and $M'$  be right and left $H$-module coalgebras, and let $N$ and $N'$ be right and left $K$-module coalgebras, respectively.  Let $\xymatrix@1{\vp :M\ar [r]&N}$  and $\xymatrix@1{\vp' :M'\ar [r]&N'}$ be DCSH-module maps with respect to $\theta$.  Then the induced chain map
$$\xymatrix{\vp\otimes_{\theta} \vp ':M\otimes _{H}M'\ar [r]&N\otimes_{K} N'}$$
is a DCSH-map.  Furthermore, if in addition $M'$ and $N'$ are right $L$-module coalgebras, where $L$ is a chain Hopf algebra, and $\vp'$ is a DCSH-module map with respect to $1_L$, then $\vp \otimes_{\theta} \vp '$ is a DCSH-module map with respect to $1_L$ as well. 
\end{thm}

\begin{proof} Colimits in $\mathbf M^\Sigma$ are calculated level-wise.  Since $\op T(A\wedge B)$ is naturally isomorphic to $\op T(A)\otimes \op T(B)$, it is easy to see that the diagrams 
$$\xymatrix{\op T(M\otimes H\otimes M')\ar @<1ex>[rr]^{\op T(1\otimes \lambda_{M'})}\ar@<-1ex>[rr]_{\op T(\rho_M\otimes 1)}&&\op T(M\otimes M')\ar [r]^{\op T(\pi)}&\op T(M\otimes_{H} M')}$$
and
$$\xymatrix{\op T(N\otimes K\otimes N')\ar@<1ex>[rr]^{\op T(1\otimes \lambda_{N'})}\ar@<-1ex>[rr]_{\op T(\rho_N\otimes 1)}&&\op T(N\otimes N')\ar [r]^{\op T(\pi)}&\op T(N\otimes_{K} N')}$$
are coequalizers in the category of $\op A$-bimodules.  On the other hand, the endofunctor $-\acirc \op F$ is a left adjoint and therefore preserves coequalizers.

Since the diagram 
$$\xymatrix{\op T(M\otimes H\otimes M')\acirc \op F\ar@<-1ex> [rr]_{\op T(\rho_M\otimes 1)\acirc 1}\ar [d]_{\widehat\vp\wedge\widehat\theta\wedge\widehat\vp'}\ar @<1ex>[rr]^{\op T(1\otimes \lambda _{M'})\acirc 1}&&\op T(M\otimes M')\acirc \op F\ar [rr]^{\op T(\pi)\acirc 1}\ar [d]_{\widehat \vp \wedge\widehat\vp '} &&\op T(M\otimes_{H} M')\acirc \op F\\
\op T(N\otimes K\otimes N')\ar @<-1ex>[rr]_{\op T(\rho _N\otimes 1)}\ar @<1ex>[rr]^{\op T(1\otimes \lambda_{N'})}&&\op T(N\otimes N')\ar [rr]_(0.5){\op T(\pi)}&&\op T(N\otimes_{K} N')}$$
of coequalizer diagrams in the category of $\op A$-bimodules commutes, there exists a map of $\op A$-bimodules
$$\xymatrix{{\widehat \vp\wedge_{ {\widehat\theta} } \widehat \vp '}:\op T(M\otimes _{H}M')\acirc \op F\ar [r]& \op T(N\otimes_{K} N')}$$
that makes the whole diagram commute. Restricting to level 1, we verify easily that
$$\widehat \vp \wedge_{{\widehat\theta}} \widehat \vp '(1)(x\otimes _{H}x'\otimes z_0)= \widehat \vp (x)\otimes _{K}\widehat\vp '(x')$$
as desired.

From the diagram above, it is easy to see that if $\vp'$ is a DCSH-module map with respect to $1_L$, then $\vp \wedge_{\theta} \vp '$ is as well.
\end{proof}

\section{DCSH-resolutions of module coalgebras}\label{subsec:resolution}

We introduce in this section the notion of \emph{DCSH-resolution} of a module coalgebra $M$ over a chain Hopf algebra $H$, as a DCSH-module map with $H$-semifree source and target $M$ that is a quasi-isomorphism.  We conclude with a lifting result for DCSH-resolutions, which proves useful in both of the applications we study later in the paper.

Let $H$ be a chain Hopf algebra, and let $M$ be an $H$-module coalgebra, i.e., $M$ admits both an $H$-action $M\otimes H \to M$ and a coassociative, counital comultiplication $M\to M\otimes M$, which is a morphism of $H$-modules.  Our goal in this section is to apply the notions of DCSH-multiplicative maps and of DCSH-module maps to defining a type of resolution of $M$ over $H$ that respects multiplicative structure exactly and comultiplicative structure up to strong homotopy.
Our extended version of Moore's theorem is expressed in terms of such highly structured resolutions.

The first step towards the definition consists in specifying the resolving objects.

\begin{defn} Let $A$ be a chain algebra, and let $M$ be a  right $A$-module.  A right $A$-module $M'$ is a \emph{ semifree extension of $M$} if it is the union of an increasing family of $A$-modules
$$M'(-1)=M\subseteq M'(0)\subseteq M'(1)\subseteq ...$$
 such that each quotient $M'(n)/M'(n-1)$ is $A$-free on a basis of cycles. 
 
 If $M=0$, then $M'$ is called a \emph{semifree $A$-module}.\end{defn}

In particular, if $M'$ is a semifree extension of $M$, then there is a free graded $\Bbbk$-module $V$ such that $M'\cong M\oplus (V\otimes A)$ as (nondifferential) graded $A$-modules.  There is an analogous notion of semifree left $A$-modules.  Note that twisted tensor products $C\otimes_{t}A$ (cf.~Appendix \ref{sec:twistingcochains}) are prime examples of semifree $A$-modules, at least when the differential graded $\Bbbk$-module underlying $A$ is itself $\Bbbk$-semifree. We refer the reader to \cite {f四ix-halperin-thomas} ﾊfor further details.

\begin{defn}\label{defn:resoln} Let $\xymatrix@1{\theta:H\ar [r]&\;K}$ be a multiplicative DCSH-map.  Let $M$ and $N$ be a right $H$-module coalgebra and a right $K$-module coalgebra, respectively, and let $\xymatrix@1{\vp :M\ar [r]&\;N}$ be a DCSH-module map with respect to $\theta$.
If $M$ is a semifree $H$-module and $\vp$ is a quasi-isomorphism, then $\vp$ is a \emph{strongly homotopy-comultiplicative resolution} or \emph{DCSH-resolution} of $N$ over $H$.
\end{defn}

\begin{ex}\label{ex:canonical-res}  Let $H$ be a connected chain Hopf algebra, with multiplication $\mu:H\otimes H\to H$, and let $M$ be a right $H$-module coalgebra, with coaction $\rho: M\otimes H\to M$.  Let $\ve: \Bar H\ot_{t_{\Bar }}H \to \Bbbk$ denote the augmentation (cf.~Appendix \ref{sec:twistingcochains}).

It follows from Theorem 4.1 in \cite{f四ix-halperin-thomas} (or the dual of Corollary 3.6 in \cite{hess-levi}) that $M\ot_{t_{\Bar}}\Bar H\ot_{\Bar}H$ admits the structure of a right $H$-module coalgebra such that the $\Bbbk$-linear map
$$\widehat \rho: M\ot_{t_{\Bar}}\Bar H\ot_{t_{\Bar}}H \to M: m\otimes w\otimes h \mapsto \ve(w) \cdot \rho(m\otimes h),$$
is a morphism of right $H$-module coalgebras.   The right $H$-action on the domain of $\widehat \rho$ is defined to be
$$(M\ot_{t_{\Bar}}\Bar H\ot_{t_{\Bar}}H)\otimes H \cong M\ot_{t_{\Bar}}\Bar H\ot_{t_{\Bar}}(H\otimes H) \xrightarrow {1\otimes \mu} M\ot_{t_{\Bar}}\Bar H\ot_{\Bar}H,$$ 
where the left action of $H$ on $H\otimes H$ is given by $\mu \otimes 1: (H\otimes H)\otimes H \to H\otimes H$. Note that  $\mu: H\otimes H \to H$ is a morphism of left $H$-modules, with respect to this action and to the usual left action of $H$ on itself.

 If, in addition, the differential graded $\Bbbk$-modules underlying $H$ and $M$ are semifree, then $\widehat \rho$ is a quasi-isomorphism and is therefore a DCSH-resolution of $M$ over $H$.
\end{ex}

In sections \ref{sec:circle} and \ref{sec:simpl-loop}, we provide further examples of DCSH-resolutions, over the chains on a topological group, in particular the circle, and over the chains on a simplicial loop group. 
The following theorem, which says that semifree extensions of module coalgebras satisfy a left lifting property with respect to surjective quasi-isomorphisms, plays an important role in both of our examples.   Section \ref{sec:longproof} is devoted to the long and technical proof of this theorem.

\begin {thm}\label{thm:exist-dcsh-res} Let $H$ be a connected chain Hopf algebra such that the underlying graded algebra is free on a free graded $\Bbbk$-module, and let $\theta:H\to K$ be a multiplicative DCSH-map.  Let $M$ and $M'$ be $H$-module coalgebras, while $N$ and $N'$ are $K$-module coalgebras.
Let 
$$\xymatrix{M\ar [d]_{j}\ar[r] ^\varphi &N\ar [d]_{p}^\simeq\\
M'\ar[r]^{\varphi'}&N'}$$
be a commuting diagram of chain maps, where
\begin{enumerate}
\item the underlying $H$-module of $M'$ is an $H$-semifree extension of $M$, and  the natural inclusion map, $j$, is strictly comultiplicative; 
\item $p$ is a surjective quasi-isomorphism of $K$-module coalgebras; and
\item $\varphi$ and $\varphi'$ are DCSH-module maps with respect to $\theta$.
\end{enumerate}

If the induced diagram in the category of $\op A$-bimodules
$$\xymatrix{\op T(M)\acirc \op F\ar [d]_{\op T(j)\acirc 1}\ar[r] ^{\widehat \varphi} &\op T(N)\ar [d]_{\op T(p)}^\simeq\\
\op T(M')\acirc \op F\ar[r]^{\widehat\varphi'}&\op T(N')}$$
commutes, then there is a DCSH-module map $\omega:M'\to N$ with respect to $\theta$, lifting $\varphi'$ and extending $\varphi$, i.e., such that $p\omega=\varphi'$ and $\omega j=\varphi$.  In particular, if $\varphi'$ is a DCSH-resolution of $N'$, then $\omega$ is a DCSH-resolution of $N$ over $H$.
\end{thm}

\section{The proof of Theorem \ref{thm:exist-dcsh-res}}\label{sec:longproof}

The following technical lemma from \cite{hess-levi} is the key to proving  Theorem \ref{thm:exist-dcsh-res}. Note that the category $\cat {M}$ of graded $\Bbbk$-modules or of chain complexes can be ``linearly'' embedded in the category $\cat M^\Sigma$ of symmetric sequences, via a functor
\begin{equation}\label{eqn:emb-lin}
 \op L:\cat M\to \cat M^\Sigma,
 \end{equation}
which is defined on objects $A$ in $\cat M$ by $\op L(A)(1)=A$ and $\op L(A)(n)=0$ for all $n\not=1$ and similarly for morphisms.  Let $\mathfrak u:\op L\to \op T$ denote the obvious ``inclusion on level $1$" natural transformation.

\begin{lem} \cite[Lemma 2.3]{hess-levi}\label{lem:inducedT} Let $A$ and $B$ be graded $\Bbbk$-modules, and let $\op X$ be a symmetric sequence of graded $\Bbbk$-modules.  Any morphism $\theta:\op L(A)\diamond \op X\to \op T(B)$ in $\cat{grMod}_\Bbbk^\Sigma$ extends naturally to a morphism $\widehat \theta: \op T(A)\diamond \op X\to \op T(B)$ of left $\op A$-modules such that $\widehat\theta(\mathfrak u\diamond 1)=\theta$.
\end{lem}

\begin{rmk}\label{rmk:scholium} It is clear from the explicit construction of $\widehat\theta$ in the  proof of Lemma \ref{lem:inducedT} that if $\theta =\varphi(\mathfrak u \diamond 1)$, where $\varphi: \op T(A)\diamond \op X \to \op T(B)$ is a map of left $\op A$-modules, then $\varphi=\widehat \theta$, i.e., the extension of $\theta$ is unique.
\end{rmk}

\begin{proof}[Proof of Theorem \ref{thm:exist-dcsh-res}] We first explain the existence of the induced diagram of $\op A$-bimodules. Since $j$ is strictly comultiplicative, it induces morphisms of $\op A$-bimodules
$$\op T(j):\op T(M)\to \op T(M')\qquad \text{and}\qquad \op T(j)\acirc 1:\op T(M)\acirc \op F\to \op T(M')\acirc \op F.$$   
Similarly, $p$ induces a morphism of $\op A$-bimodules
$$\op T(p):\op T(N)\to \op T(N').$$ 
Moreover, that $\varphi$ and $\varphi'$ are DCSH-module maps means that there are morphisms of $\op A$-bimodules
$$\widehat\varphi:\op T(M)\acirc \op F \to \op T(N)\qquad\text{and}\qquad \widehat\varphi' :\op T(M')\acirc \op F \to \op T(N')$$
satisfying coherence conditions as in Definition \ref{defn:dcsh-module}.

To prove the theorem, it suffices to consider the case $M'/M\cong (\Bbbk\cdot v)\otimes H$, i.e., the case in which $M'$ is obtained from $M$ by adjoining one new generator $v$, with $dv\in M$ and $\overline\Delta (v)\in M\otimes M$, where $d$ and $\overline \Delta$ are the differential and the reduced comultiplication on $M'$.  The general case then follows by an inductive argument, since we can pass from the finitely generated case to the case of an arbitrary semifree extension by taking directed colimits.
Assume therefore that 
$$M'/M\cong (\Bbbk\cdot v)\otimes H.$$ 

Let $TW$ denote the free graded algebra underlying $H$.  To construct $\omega$, we proceed by induction on the degree of the generators of $TW$ and on the degree of the $\op A$-bimodule generators of $\op F$ (cf., Remark \ref{rmk:filtered-structure}).   Much of the proof closely resembles standard lifting results for chain complexes, chain algebras, etc., but we have to be a little bit careful in order to ensure that our lift is sufficiently highly structured.

Let $\op F_{n}$ denote the sub $\op A$-co-ring of $\op F$ freely generated as an $\op A$-bimodule by $\{z_{k}\mid k<n\}$. Note that the image of the restriction of $\Delta _{\op F}$ to $\op F_{n}$ lies in $\op F_{n}\otimes \op F_{n}.$

For any positive integer $n$, let $H_{(n)}$ denote the sub Hopf algebra of $H$ (freely) generated as an algebra by $W_{<n}$, and let $H_{(0)}=0$.  We thus have   an increasing filtration of $H$
$$H_{(0)}=0\subset H_{(1)}=\Bbbk\subset \cdots \subset H_{(n)}\subset H_{(n+1)}\subset \cdots ,$$
which induces an increasing filtration of $M'$
$$M'_{(0)}=M\subset M'_{(1)}=M\oplus \Bbbk\cdot v\subset \cdots \subset M'_{(n)}\subset M'_{(n+1)}\subset \cdots .$$

Define bigraded families of symmetric sequences of chain complexes 
$$\{\op X_{n,m}\mid n,m\in \mathbb N\}\text{ and }\{\op Y_{n,m}\mid n,m\in \mathbb N\}$$ 
by 
$$\op X_{n,m}=\op T(M)\acirc \op F + \op T(M')\acirc \op F_{n}+\op T(M'_{(m)} )\acirc \op F_{n+1}$$
and
$$\op Y_{n,m}=\op T(M\otimes H)\acirc \op F + \op T(M'\otimes H)\acirc \op F_{n}+\op T(M'_{(m)} \otimes H_{(m)})\acirc \op F_{n+1}$$
Observe that
\begin{itemize} 
\item $\op X_{0,0}=\op T(M)\acirc \op F$;
\item $\operatorname{colim}_{m}\op X_{n,m}=\op X_{n+1,0}=\op T(M)\acirc \op F+\op T(M')\acirc \op F_{n+1}$;
\item $\operatorname{colim}_{n,m}\op X_{n,m}=\op T(M')\acirc \op F$.
\end{itemize}
For all $m\leq m', n\leq n'$, let $\iota _{n,m}^{n',m'}:\op X_{n,m}\to \op X_{n',m'}$ denote the inclusion.

Let $\cat{L}_{n,m}$ denote the following statement.
\medskip
\begin{quote}
There is a morphism of $\op A$-bimodules in the category of symmetric sequences of chain complexes $$\widehat \omega_{n,m}:\op X_{n,m}\to \op T(N)$$ 
such that
\begin{enumerate}
\item $\widehat \omega_{n,m}$ extends $\widehat \varphi$, i.e., $\widehat \omega _{n,m}\circ\iota _{0,0}^{n,m}=\widehat\varphi$;
\item $\widehat \omega_{n,m}$ lifts $\widehat \varphi'$, i.e., $\op T(p)\circ \widehat\omega_{n,m}=\widehat\varphi'|_{\op X_{n,m}}$; and
\item  the following diagram commutes 
$$\xymatrix{\op Y_{n,m}\ar [rr]^{\op T(\rho _{M'})\acirc 1}\ar [d]^{\widehat\omega_{n,m}\wedge \widehat\theta}&&\op X_{n,m}\ar [d]^{\widehat\omega_{n,m}}\\
\op T(N\otimes K)\ar [rr]^{\mathcal T(\rho _N)}&&\op T(N)}.$$
 Here, $\widehat\omega _{n,m}\wedge \widehat \theta$ is defined in terms of the restriction of $\Delta_{\op F}$ to $\op F_{n}$ or to $\op F_{n+1}$, in a slight variation on the diagram (\ref{eqn:monoidal-fatcoalg}).
\end{enumerate}
\end{quote}
\medskip

We show below that 
\begin{equation}\label{eqn:inductive-step}
\cat L_{n,m}\Longrightarrow \cat L_{n,m+1}
\end{equation}
for all $n,m\in \mathbb N$ and $\widehat\omega_{n,m}=\widehat\omega_{n,m+1}\circ \iota_{n,m}^{n,m+1}$.  It  follows by induction that 
\begin{equation}\label{eqn:n-to-n+1}\cat L_{n,0}\Longrightarrow \cat L_{n+1,0},\end{equation}
since we can set
$$\widehat \omega_{n+1,0}=\operatorname{colim}_{m}\widehat\omega_{n,m}:\op X_{n+1,0}=\operatorname{colim}_{m}\op X_{n,m}\to \op T(N).$$
Since $\cat L_{0,0}$ certainly holds, where $\widehat\omega_{0,0}=\widehat\varphi$, it follows from (\ref{eqn:n-to-n+1})  by another induction that $\cat L_{n,0}$ holds for all $n$.  We can therefore set
$$\widehat\omega=\operatorname{colim}_{n}\widehat\omega _{n,0}:\op T(M')\acirc \op F=\operatorname{colim}_{n}\op X_{n,0}\to \op T(N)$$
and obtain a morphism of $\op A$-bimodules such that
$$\xymatrix{\op T(M)\acirc \op F\ar [d]_{\op T(j)\acirc 1}\ar[r] ^{\widehat \varphi} &\op T(N)\ar [d]_{\op T(p)}^\simeq&\text{and}&\op T(M'\otimes H)\acirc \op F\ar [rr]^(0.6){\op T(\rho _{M'})\acirc 1}\ar [d]^{\widehat\omega\wedge \widehat\theta}&&\op X_{n,m}\ar [d]^{\widehat\omega}\\
\op T(M')\acirc \op F\ar[r]^{\widehat\varphi'}\ar[ur]^{\widehat\omega}&\op T(N')&&\op T(N\otimes K)\acirc \op F\ar[rr]^{\op T(\rho_{N})}&&\op T(N)}$$
commute.
It follows that if $\omega:M\to N$ denotes the restriction of $\widehat\omega$ to $\op T(M)(1)$,  then $\omega$ is a DCSH-module map with respect to $\theta$.  Thus, once we have proved (\ref{eqn:inductive-step}), the theorem itself will have been proved as well.

We now prove (\ref{eqn:inductive-step}). To simplify notation, let $D$ denote the differential in all of the symmetric sequences of chain complexes that we consider. 

Suppose that $\cat L_{n,m}$ holds. Let $w\in W_{m}$. Consider $$(v\otimes w)\otimes z_{n}\in \op T(M'_{(m+1)})\acirc \op F_{n+1}.$$ Since $dv\in M$ and $\overline\Delta (v)\in M\otimes M$, the formula for the differential $\del_{\op F}$ (cf., \cite [p.853]{hpst}) implies that 
$$D\big((v\otimes w)\otimes z_{n}\big )\in \op X_{n,m}$$
and therefore that $\widehat \omega_{n,m} D\big((v\otimes w)\otimes z_{n}\big )$ is defined.  Moreover, 
$$D\widehat \omega_{n,m} D\big((v\otimes w)\otimes z_{n}\big )=\widehat \omega_{n,m} D^2\big((v\otimes w)\otimes z_{n}\big )=0.$$

Since $\op T(p)$ is surjective, there exists $x\in \op T(N)(n+1)$ such that 
$$\op T(p)(x)=\widehat \vp '\big((v\otimes w)\otimes z_{n}\big ).$$
Note that 
$$\op T(p)\left (Dx-\widehat \omega_{m,n} D\big((v\otimes w)\otimes z_{n}\big )\right)=D\op T(p)(x)-\widehat \vp'D\big((v\otimes w)\otimes z_{n}\big )=0.$$
It follows that $Dx-\widehat \omega _{n,m}D\big((v\otimes w)\otimes z_{n}\big ) $ is a cycle in the kernel of $\op T(p)(n+1)$, which is acyclic, since $\op T(p)$ is a surjective quasi-isomorphism in each level.  There exists therefore $y\in \ker \op T(p)(n+1)$ such that $Dy=Dx-\widehat \omega D\big((v\otimes w)\otimes z_{n}\big )$.  We can thus set
$$\widehat\omega_{n,m+1}\big((v\otimes w)\otimes z_{n}\big )=x-y\in \op T(N)(n+1)$$
and ensure that $\widehat \omega_{n,m} D\big((v\otimes w)\otimes z_{n}\big )=D\widehat \omega_{n,m+1} \big((v\otimes w)\otimes z_{n}\big )$.  Assume henceforth that we have made such a choice for each element of a basis of the free $\Bbbk$-module $W_{m}$ and then extended $\Bbbk$-linearly to all of $W_{m}$.

Let $\op S_{n+1}$ be the sub symmetric sequence of $\op S$ (cf.~Remark \ref{rmk:filtered-structure}) such that $\op S_{n+1}(k)=\op S(k)$ for $k\leq n+1$ and $\op S_{n+1}(k)=0$ for $k>n+1$. There is a map of symmetric sequences of chain complexes
\begin{equation}\label{eqn:linear-map}
\op L\big (M'_{(m)}\oplus(\Bbbk\cdot v\otimes W_{m})\big)\diamond \op S_{n+1}\to \op T(N)
\end{equation} 
defined on $\op L(M'_{(m)})\diamond \op S_{n+1} \oplus \op L( \Bbbk\cdot v\otimes W_{m})\diamond \op S_{n}$ to be the restriction of the map $\widehat\omega_{n,m}$ and on $(v\otimes w)\otimes z_{n}$ to be $\widehat \omega_{n,m+1} \big((v\otimes w)\otimes z_{n}\big)$ for all $w\in W_{m}$.

  By Lemma \ref{lem:inducedT}, the map (\ref{eqn:linear-map}) induces a unique map of left $\op A$-modules
\begin{equation}\label{eqn:left-amod-map}
\op T\big (M'_{(m)}\oplus(\Bbbk\cdot v\otimes W_{m})\big)\diamond \op S_{n+1}\to \op T(N)
\end{equation}
that agrees with the restriction of $\widehat\omega_{n,m}$ on $\op T(M'_{(m)})\diamond \op S_{n+1}$ and on  $\op T\big(M\oplus (\Bbbk\cdot v\otimes W_{m})\big)\diamond S_{n}$.
Since the underlying $\op A$-bimodule of $\op F_{n+1}$ is free on $\op S_{n+1}$, the map (\ref{eqn:left-amod-map}) induces a map of $\op A$-bimodules
\begin{equation}\label{eqn:right-amod-map}
\widehat\omega _{n,m+1}^1:\op T\big (M'_{(m)}\oplus(\Bbbk\cdot v\otimes W_{m})\big)\acirc \op F_{n+1}\to \op T(N)
\end{equation}
that agrees with the restriction of $\widehat\omega_{n,m}$ on $\op T(M'_{(m)})\acirc\op F_{n+1}$ and on $\op T\big (M'_{(m)}\oplus(\Bbbk\cdot v\otimes W_{m})\big)\acirc \op F_{n}$.  In particular,  condition (1) of statement $\cat L_{n,m+1}$ holds and $\widehat\omega_{n,m}=\widehat\omega_{n,m+1}\circ \iota_{n,m}^{n,m+1}$.  Moreover, by construction, condition (2) of statement $\cat L_{n,m+1}$ is also satisfied.

Since $H$ is freely generated as an algebra by $W$, which itself is a free $\Bbbk$-module, we can extend the map $\widehat\omega _{n,m+1}^1$ to all of  $\op T(M'_{(m+1)})\acirc \op F_{n+1}$ in such a way that condition (3) of statement $\cat L_{n,m}$ is satisfied.  We proceed by induction on wordlength in the free algebra $H_{(m+1)}=TW_{\leq m}$.  

Let $H_{(m+1)}^k$ denote the subcomplex of $H_{(m+1)}$ generated by  words that have at most $k$ letters coming from $W_{m}$.  Let 
$$\op X_{n, m+1}^k =\op X_{n,m}+\op T\big (M'_{(m)}\oplus(\Bbbk\cdot v\otimes H_{(m+1)}^k)\big)\acirc \op F_{n+1}$$
and
$$\op Y^k_{n,m+1}=\op Y_{n,m}+\sum _{i+j=k}\op T\Big(\big (M'_{(m)}\oplus(\Bbbk\cdot v\otimes H_{(m+1)}^i)\big)\otimes H_{(m+1)}^j\Big).$$
Let $\cat{H}_{n,m+1,l}$ denote the following statement.
\medskip
\begin{quote}
There are morphisms of $\op A$-bimodules in the category of symmetric sequences of chain complexes 
$$\widehat \omega_{n,m+1}^{k}:\op X_{n,m+1}^k\to \op T(N), \forall k\leq l$$ 
such that
\begin{enumerate}
\item for all $k,k'\leq l$, $\widehat \omega_{n,m+1}^{k}$ and $\widehat \omega_{n,m+1}^{k'}$ agree on the intersection of their domains;
\item each $\widehat \omega_{n,m+1}^k$ extends $\widehat \varphi$;
\item each $\widehat \omega_{n,m+1}^k $ lifts $\widehat \varphi'$, i.e., $\op T(p)\circ \widehat\omega_{n,m+1}^k=\widehat\varphi'$; and
\item  the following diagram commutes 
$$\xymatrix{\op Y^k_{n,m+1}\ar [rr]^{\op T(\rho _{M'})\acirc 1}\ar [d]^{\chi_{n,m+1}^k}&&\op X^k_{n,m+1}\ar [d]^{\widehat\omega_{n,m+1}^k}\\
\op T(N\otimes K)\ar [rr]^{\mathcal T(\rho _N)}&&\op T(N)},$$
for all $k\leq l$, where 
$$\chi_{n,m+1}^k=\widehat \omega_{n,m}\wedge \theta+\sum _{i+j=k} \widehat \omega_{n,m+1}^{i}\wedge \theta|_{H^j_{m+1}},$$
which is defined in terms of the restriction of $\Delta_{\op F}$ to $\op F_{n}$ or to $\op F_{n+1}$, in a slight variation on the diagram (\ref{eqn:monoidal-fatcoalg}).
\end{enumerate}
\end{quote}
We have shown that $\cat{H}_{n,m+1,1}$ holds. To complete the proof that (\ref{eqn:inductive-step}) holds, it suffices to prove that $\cat{H}_{n,m+1,l}$ implies $\cat{H}_{n,m+1,l+1}$ for all $l$, since 
$$\op X_{n, m+1} =\operatorname{colim}_{l}\op X_{n,m+1}^l \quad \text{ and }\quad \op Y_{n, m+1} =\operatorname{colim}_{l}\op Y_{n,m+1}^l.$$
We leave the details of the inductive step to the reader, as it proceeds essentially identically to the argument above, using acyclicity of the kernel of $\op T(p)$ to choose an image for each element of $\op X_{n,m}^{l+1}$ of the form
$$(v\otimes a)\otimes z_{n},$$
where $a$ is a basis element of $H_{(m+1)}^{l+1}$, then calling on Lemma \ref{lem:inducedT}.
\end{proof}

\section{Comultiplicative enrichment of Moore's theorem}

The goal of this section is to apply DCSH-resolutions of Hopf algebras and of module coalgebras over Hopf algebras to enriching  Moore's Theorem (cf.~Introduction), obtaining an isomorphism that preserves natural comultiplicative structure.

Let $G$ be a connected topological group, let $E$ be the total space of a principal $G$-bundle, where $G$ acts on $E$ on the right,  and let $Y$ be a left $G$-space.  Let $r: E\times G\to E$ and $l:G\times Y\to Y$ be the actions.  Let $p:E\times Y\to E\times_G Y$ denote the quotient map.

Recall that for any pair of spaces $X$ and $W$,  the Eilenberg-Zilber (or shuffle) equivalence $\xymatrix@1{\text{EZ}:C_*X\otimes C_*W\ar [r]&C_*(X\times W)}$ is a coalgebra map.  Consequently, the induced maps
$$\xymatrix{C_*E\otimes C_*G\ar [r]^{\text{EZ}}_\simeq\ar@/_1.5pc/ [rr]_\rho&C_*(E\times G)\ar [r]^(0.6){C_*r}&C_*E}$$
and
$$\xymatrix{C_*G\otimes C_*Y\ar [r]^{\text{EZ}}_\simeq\ar@/_1.5pc/ [rr]_\lambda&C_*(G\times Y)\ar [r]^(0.6){C_*l}&C_*Y}$$
are coalgebra maps as well.
As observed in the previous section, the chain complex $C_*E \otimes_{C_*G} C_*Y$ therefore admits a coalgebra structure such that the quotient map $\pi:  C_*E\otimes C_*Y\to C_*E \otimes_{C_*G} C_*Y$ is a coalgebra map.  Furthermore, the chain map $C_*E \otimes_{C_*G} C_*Y\to  C_*(E\times_G Y)$ induced by the composite
\begin{equation}\label{eqn:coalg-map}
\xymatrix{C_*E\otimes C_*Y\ar [r]^{\text {EZ}}&C_*(E\times Y)\ar [r]^{C_*p}& C_*(E\times _G Y)}
\end{equation}
is also a coalgebra map.

We now use the results above on DCSH-resolutions to prove a more highly structured version of Moore's classic theorem \cite {moore}.  In the proof we make extensive use of twisting cochains; we refer the reader to Appendix \ref{sec:twistingcochains} for basic definitions, notation and examples.

\begin{thm}\label{thm:enriched-moore} Let $G$ be a connected topological group, let $E$ be the total space of a principal $G$-bundle, where $G$ acts on $E$ on the right,  and let $Y$ be a left $G$-space. Let $\xymatrix@1{\theta:H\ar [r]&\;C_*G}$ be a multiplicative DCSH quasi-isomorphism, where $H$ is connected, and the algebra underlying $H$ is free on a free graded $\Bbbk$-module.  

If $\xymatrix@1{\vp :M\ar [r]^\simeq&\;C_*E}$ is a DCSH $H$-resolution of $C_*E$, then there is a DCSH quasi-isomorphism
$$\xymatrix{M\otimes_{H} C_{*}Y\ar [r]^{\simeq} &C_{*}(E \times_{G} Y).}$$
 In particular, $$H^*\bigl((M\otimes_{H} C_*Y)^\sharp\bigr)\cong H^*(E\times_G Y)$$
as graded algebras.
\end{thm}

\begin{proof} Recall that Moore proved in \cite {moore} that given any $C_*G$-semifree resolution of $C_*E$, 
$$\xymatrix{\psi: N\ar [r]^{\simeq}&C_*E}$$
the composite
$$\xymatrix{N \otimes_{{C_*G}} C_*Y\ar [r]^{\psi\otimes 1}&C_*E \otimes_{C_*G} C_*Y\ar [r]& C_*(E\times_G Y)}$$
is a quasi-isomorphism.

It follows from Example \ref{ex:canonical-res} that we can apply Theorem \ref{thm:exist-dcsh-res} to the diagram
$$\xymatrix{0\ar [d]\ar [r]& C_{*}E\ot_{t_{\Bar}} \Bar C_{*}G\ot_{t_{\Bar}} C_{*}G\ar@{->>}[d]_{\simeq}^{\widehat \rho}\\
M\ar [r]^\vp_{\simeq}&C_{*}E,}$$
obtaining a DCSH-module map with respect to $\theta$,
$$\xymatrix{\omega:M \ar [r]^(0.3)\simeq &C_{*}E\ot_{t_{\Bar}} \Bar C_{*}G\ot_{t_{\Bar}} C_{*}G.}$$
Theorem \ref{thm:tensor-dcsh} then implies that 
$$\xymatrix{\omega\ot_{\theta} 1_{C_{*}G}:M\otimes _{H}C_{*}G \ar [r]^(0.5)\simeq &C_{*}E\ot_{t_{\Bar}} \Bar C_{*}G\ot_{t_{\Bar}} C_{*}G}$$
is a DCSH-map, which is also a quasi-isomorphism by Proposition 2.4 in \cite{f四ix-halperin-thomas}.  Composing with $\widehat \rho$, which is a strictly comultiplicative map,  gives rise to a DCSH-quasi-isomorphism
$$\xymatrix{\widehat \omega :=\widehat \rho\circ (\omega\ot_{\theta} 1_{C_{*}G}):M\otimes _{H}C_{*}G \ar [r]^(0.75)\simeq &C_{*}E,}$$
to which we can apply Moore's result, since $M\otimes _{H}C_{*}G$ is $C_{*}G$-semifree.  The composite
$$\xymatrix{M\ot_{H}C_{*}Y\cong (M\otimes _{H}C_{*}G)\ot_{C_{*}G}C_{*}Y\ar[rr]^(0.65){\widehat \omega\ot 1_{C_{*}Y}}&& C_{*}E\ot_{C_{*}G} C_{*}Y\ar[d]^{q}\\
&& C_{*}(E\times_{G}Y)}$$
is therefore a quasi-isomorphism.  Since $ \widehat \omega\ot 1_{C_{*}Y}$ is a DCSH-map by Theorem \ref{thm:tensor-dcsh}, and $q$ is strictly commutative, the composite is a DCSH-map, and we have the desired resolution.
\end{proof} 

There is also a simplicial version of Theorem \ref{thm:enriched-moore}, of which the proof is essentially identical.

\begin{thm}\label{thm:enriched-moore-simpl} Let $G$ be a reduced simplicial group, let $E$ be the total space of a principal twisted cartesian product with group $G$, where $G$ acts on $E$ on the right,  and let $L$ be a simplicial set admitting a left $G$-action. Let $\xymatrix@1{\theta:H\ar [r]&\; C_*G}$ be a multiplicative DCSH quasi-isomorphism, where $H$ is  connected, and the algebra underlying $H$ is free on a free graded $\Bbbk$-module.  Let $\xymatrix@1{\vp :M\ar [r]&\; C_*E}$ be a DCSH $H$-resolution of
 $C_*E$.  Then there is a DCSH quasi-isomorphism
$$M\underset H\otimes C_{*}L\xrightarrow{\simeq} C_{*}(E \underset G\times L).$$
 In particular, $$H^*\bigl((M\underset H\otimes C_*L)^\sharp\bigr)\cong H^*(E\underset G\times L)$$
as graded algebras.
\end{thm}

We refer the reader to \cite{may} ﾊfor the definition of twisted cartesian products of simplicial sets.


\section{Homotopy orbits of circle actions}\label{sec:circle}

Let $Y$ be a topological space endowed with a left action of the circle $S^1$.  If $ES^1$ is a contractible, free $S^1$-space, then a model of the \emph{ homotopy orbit space} of $Y$, denoted $Y_{hS^1}$,  is  $ES^1\underset {S^1}\times Y$.  

In this section we explain how to construct a model for the cohomology algebra of $Y_{hS^1}$  by applying our enriched version of  Moore's theorem (Theorem \ref{thm:enriched-moore}).  We begin by finding a particularly nice family of primitive elements in $CU_*S^1$, which we proceed to apply to building a highly structured resolution of $CU_*ES^1$, where $CU_{*}$ denotes the cubical chains functor.  Using that resolution,  we then obtain the desired model for $Y_{hS^1}$ as a consequence of Theorem \ref{thm:enriched-moore}.

\subsection {A special family of primitives}

In this section, as a first step towards defining a DCSH-resolution of $CU_*ES^1$, we identify an interesting family of primitive elements in $CU_*S^1$. We begin by defining a suspension-type degree +1 operation on $CU_*S^1$.

\begin{defn}If $\xymatrix@1{T:I^n\ar [r]&S^1}$ is an $n$-cube, let $\sigma (T)$ be the $(n+1)$-cube defined by  
$$\sigma (T)(t_0,...,t_{n}):=\bigl(T(t_1,...,t_n)\bigr)^{t_0},$$
where we are considering $S^1$ as the unit circle in the complex plane, i.e.,
$$S^1=\{e^{i\theta}\mid 0\leq \theta\leq 2\pi\}.$$\end{defn}

\begin{rmk} It is clear that $\sigma (T)$ is degenerate if $T$ is degenerate.  The operation $\sigma$ can therefore be extended linearly to all of $CU_*S^1$.\end{rmk}

As the next lemma states, $\sigma $ is a contracting homotopy in degrees greater than one and is  a $(1,0)$-coderivation.

\begin{lem}\label{lem:t-defn} Let $T\in CU_*S^1$.
\begin{enumerate}
\item If $\deg T\geq 2$, then $d\sigma (T)=T-\sigma (dT)$ where $d$ is the usual differential on $CU_*S^1$.
\item $\overline\delta_{S^1} (\sigma (T))=\sigma (T_i)\otimes T^i$, where $\overline\delta_{S^1}$ is the usual reduced coproduct on $CU_*S^1$ and $\delta_{S^1} (T)=T_i\otimes T^i$ (using the Einstein summation convention).
\end{enumerate}
\end{lem}

Simple calculations, applying the definitions of the cubical differential and the cubical coproduct, as given for example in \cite {massey} and \cite {anick}, suffice to prove this lemma.

We now apply the $\sigma $ operation to the recursive construction of a certain family of elements in $CU_*S^1$.

\begin{defn}\label{defn:primfamily}Let $\xymatrix@1{T_0:I\ar [r]&S^1}$ be the $1$-cube defined by $T_0(t)=e^{i2\pi t}$.  Given $T_k\in CU_{2k+1}S^1$ for all $k<n$, let $T_n$ be the $(2n+1)$-cubical chain defined by 
$$T_n:=\sigma \biggl(\sum _{i=1}^n T_{i-1}\cdot T_{n-i}\biggr)\in CU_{2n+1}S^1.$$
Let $\mathfrak  T:=\{T_n\mid n\geq 0\}$.\end{defn}

\begin{ex} It is easy to see that
$$T_1(t_0, t_1, t_2)=e^{i2\pi t_0(t_1+t_2)}$$
and that 
$T_2=U+V$ where
$$U(t_0,...,t_4)=e^{i2\pi t_0(t_1+(t_2+t_3)t_4)}\text{ and }V(t_0,...,t_4)=e^{i2\pi t_0(t_1(t_2+t_3)+t_4)}.$$
\end{ex}

\begin{prop}\label{prop:t-properties} The family $\mathfrak  T$ satisfies the following properties.
\begin{enumerate}
\item $dT_0=0$, and $0\not=[T_0]$ in $H_1S^1$.
\item $dT_n=\sum _{i=1}^n T_{i-1}T_{n-i}$ for all $n>0$.
\item Every $T_n$ is primitive in $CU_*S^1$.
\end{enumerate}
\end{prop}

\begin{proof} Points (1) and (2) are immediate consequences of Lemma \ref{lem:t-defn}.  It is well known that $T_0$ represents the unique nonzero homology generator of $H_*S^1$.

An easy inductive argument applying Lemma \ref{lem:t-defn}(2) proves point (3), since if $T_k$ is primitive for all $k<n$, then the sum $\sum _{i=1}^n T_{i-1}\cdot T_{n-i}$ is also primitive, as it is symmetric and all factors are of odd degree.\end{proof}

Let $\mathbb  T$ denote the subalgebra of $CU_*S^1$ generated by the family $\mathfrak   T$.  Since all the $T_n$'s are primitive, $\mathbb  T$ is a sub Hopf algebra of $CU_*S^1$. Proposition \ref{prop:t-properties} (1) and (2) imply that $\mathbb  T$ is closed under the differential.

It is helpful to recognize $\mathbb  T$ as the image of a certain homomorphism, as we next make explicit. Let $\Gamma $ denote the divided powers algebra functor.  If $v$ is in even degree, then
$$\Gamma v=\bigoplus _{n\geq 0}\Bbbk \cdot v(n),$$
where $|v(n)|=n\cdot |v|$, $v(0)=1$, $v(1)=w$ and $v(k)v(l)=\binom {k+l}{k}v(k+l)$.  Furthermore, $\Gamma v$ is in fact a Hopf algebra, where the coproduct is specified by $\Delta (v)=v\otimes 1 +1\otimes v$, which in turn implies that for all $n\geq 1$,
$$\Delta \big(v(n)\big) =\sum _{k=0}^n v(k)\otimes v(n-k).$$   
In particular, $\Delta$ is cocommutative.  Note that the $\Bbbk$-dual of $\Gamma v$ is the free, commutative algebra $\Lambda v^\sharp$ on the $\Bbbk$-linear functional $v^\sharp: \Bbbk\cdot v \to \Bbbk$ sending $v$ to $1$. 

Recall that the  homology $\H _*BS^1$ of the classifying space of the circle is isomorphic as a Hopf algebra to $(\Gamma v, \Delta)$, where $v$ is of degree 2.  
Define a linear map $\zeta :\si \H_*BS^1\to  CU_*S^1$ by $\zeta (\si v(k))=T_{k-1}$.  A simple calculation, based on Proposition \ref{prop:t-properties} (1) and (2), shows that $\zeta $ extends to a quasi-isomorphism of chain Hopf algebras 
$${\widehat\zeta}: \Om \H_*BS^1 \xrightarrow{\simeq}CU_*S^1,$$
where $\Om \H _*BS^1$ is primitively generated.
It is clear that $\mathbb  T=\operatorname {Im}\widehat \zeta$.

\subsection {Modeling $S^1$-homotopy orbits}

Using the family $\mathfrak  T$, we now construct a DCSH-resolution of $CU_*ES^1$ as a $\Om \H_*BS^1$-module.  

Let $\H_*BS^1\otimes _{t_{\Om}}\Om \H_*BS^1$ denote the acyclic cobar construction on $\H_*BS^1$ (Example \ref{ex:acyc-bar}).  Explicitly, $\H_*BS^1\otimes _{t_{\Om}}\Om \H_*BS^1=(\Gamma v\otimes T\si \Gamma ^+v, D_\Om)$, where $\Gamma^+ v= \bigoplus_{n\geq 1}\Bbbk \cdot v(n) $, and  
$$D_\Om (v(n)\otimes w)=v(n)\otimes d_\Om w-\sum _{i=0}^{n-1} v(i)\otimes \si v(n-i)\cdot w$$
for all $n$ and for all $w\in T\si \Gamma ^+v$.   Since the comultiplication $\Delta$ on $\H_{*}BS^1$ is cocommutative, it is a map of coalgebras and therefore induces a comultiplication $\psi$ on $\Om \H_{*}BS^1$ equal to the composite 
$$\Om \H_{*}BS^1 \xrightarrow{\Om \Delta}\Om (\H_{*}BS^1\otimes \H_{*}BS^1) \xrightarrow {q} \Om \H_{*}BS^1\ot \Om \H_{*}BS^1,$$
which is a map of chain algebras, where $q$ is Milgram's chain algebra quasi-isomorphism \cite{milgram}, given by 
$$q\big(\si (w\ot w')\big) =\begin{cases} \si w \ot 1 &: w'=1\\ 1\ot \si w'&: w=1\\ 0&: \text{else.}\end{cases}$$
In particular, $q\circ \Om \Delta \big(\si v(n)\big) =\si v(n) \otimes 1 + 1\otimes \si v(n)$ for all $n\geq 1$, i.e., the Hopf algebra $\Om \H_{*}BS^1$ is primitively generated.  A straightforward calculation shows that $\psi$ extends to a differential comultiplication $\widehat \psi$ on $\H_*BS^1\otimes _{t_{\Om}}\Om \H_*BS^1$ given by
$$\widehat\psi \big (v(n)\otimes w\big)=\sum _{k=0}^n  \big(v(k)\otimes w_{i}\big) \otimes \big( v(n-k) \otimes w^{i}\big),$$
where  $\psi(w)=w_{i}\otimes w^{i}$ (using Einstein summation notation).

Let $\xymatrix@1{j:S^1\ar [r]&\; ES^1}$ denote the inclusion of $S^1$ as the base of Milnor's construction of $ES^1$, which is an $S^1$-equivariant map. The composite
$$\xymatrix{CU_*j\circ \widehat\zeta:  \Om \H_*BS^1\ar [r]&CU_*ES^1}$$  
is map of $\Om \H_*BS^1$-module coalgebras. Consider the following commutative diagram of right $\Om \H_*BS^1$-module coalgebras.
\begin{equation}\label{eqn:diag-circle}
\xymatrix{\Om \H_*BS^1\ar [d]^{\iota}\ar [r]^{CU_*j\circ \widehat \zeta}& CU_*ES^1\ar @{->>}[d]_\simeq\\
\H_*BS^1\otimes _{t_{\Om}}\Om \H_*BS^1\ar [r]^(.7){\simeq}&\mathbb  Z}
\end{equation}
The inclusion $\iota$ is a semifree extension of $\Om H_{*}BS^1$-module coalgebras, and the other vertical arrow is a surjective quasi-isomorphism, while the two horizontal maps are strict maps of $\Om H_{*}BS^1$-module coalgebras.  We can therefore apply Theorem \ref{thm:exist-dcsh-res} to diagram (\ref{eqn:diag-circle}) and obtain  a DCSH $\Om \H_*BS^1$-resolution of $CU_*ES^1$:
\begin{equation}\label{eqn:dcsh-resoln}
\xi: \H_*BS^1\otimes _{t_{\Om}}\Om \H_*BS^1\xrightarrow{\simeq}CU_*ES^1.
\end{equation}

Theorem \ref{thm:enriched-moore} applied to the DCSH-resolution (\ref{eqn:dcsh-resoln}) implies the existence of a chain coalgebra model for $S^1$-homotopy orbits, as stated precisely below, where we use that
$$(\H_*BS^1\otimes _{t_{\Om}}\Om \H_*BS^1)\ot_{\Om \H_*BS^1} CU_{*}Y\cong \H_*BS^1\otimes _{\widehat \zeta\circ t_{\Om}} CU_*Y$$
(cf.~Definition \ref{defn:ttp}).

\begin{thm}\label{thm:moore-circle} Let $Y$ be any left $S^1$-space.  Then there is a DCSH quasi-isomorphism
$$ \H_*BS^1\otimes _{\widehat \zeta\circ t_{\Om}} CU_*Y\xrightarrow{\simeq } CU_{*}Y_{hS^1}.$$
In particular,
 $$\H^*\biggl(\bigl(\H_{*}BS^1\otimes _{\widehat \zeta\circ t_{\Om}} CU_*Y\bigr)^\sharp\biggr)\cong \H^*(Y_{hS^1})$$ as graded algebras.\end{thm}
 
 Applying Theorem \ref{thm:moore-circle} to the case where $Y$ is a one-point space, we obtain the following amusing corollary.
 
 \begin{cor}\label{cor:bs1}There is a DCSH quasi-isomorphism $(H_{*}BS^1,0)\xrightarrow\simeq CU_{*}BS^1$.
 \end{cor}

We now describe explicitly the model  $\H_*BS^1\otimes _{\widehat \zeta\circ t_{\Om}} CU_*Y$ of $CU_{*}Y_{hS^1}$.  Recall the family $\mathfrak T$ of primitives in $CU_{*}S^1$ (Definition \ref{defn:primfamily}). Let 
$\xymatrix@1{g:S^1\times Y\ar [r]&Y}$ be the action map.  Let $\kappa$ denote the composite
$$\xymatrix{CU_*S^1\otimes CU_*Y\ar [r]^{\text {EZ}}_\simeq\ar@/_1.5pc/ [rr]_\kappa&CU_*(S^1\times Y)\ar [r]^(0.6){CU_*g}&CU_*Y}.$$

Let $D$ denote the differential of $\H_*BS^1\otimes _{\widehat \zeta\circ t_{\Om}} CU_*Y$, and let $\delta _{Y}$ denote the usual cubical comultiplication on $CU_{*}Y$. The formula in Definition \ref{defn:ttp} for the differential of a twisted tensor product implies that for all $n\geq 0$ and all $U\in CU_{*}Y$,
$$D \big( v(n)\otimes U)= v(n)\otimes  dU - \sum _{k=0}^{n-1} v(k)\otimes \kappa (T_{n-k-1}\otimes U).$$
Moreover, the comultiplication $\widetilde \psi$ induced on $\H_*BS^1\otimes _{\widehat \zeta\circ t_{\Om}} CU_*Y$ by those of 
$\H_*BS^1\otimes _{t_{\Om}}\Om \H_*BS^1$ and of $CU_{*}Y$ is given simply by
$$\widetilde \psi \big(v(n)\otimes U\big )=\sum _{k=0}^n\big(v(k)\otimes U_{i}\big)\otimes \big( v(n-k) \otimes U^{i}\big),$$
where $\delta_{Y} (U)=U_{i}\otimes U^{i}$ (using the Einstein summation convention).

Note that this model fits into a commutative diagram
$$\xymatrix{CU_*Y\ar[d]^=\ar[r]^(0.4){}&(\Gamma v\otimes CU_*Y,D)\ar[d]^\simeq\ar[r]^(0.6){}& (\Gamma v,0)\ar[d]^\simeq\\
CU_*Y\ar[r]&CU_*Y_{hS^1}\ar[r]&CU_*BS^1,}$$
where $v$ is of degree $2$, and the rightmost vertical arrow is the DCSH quasi-isomorphism of Corollary \ref{cor:bs1}.

If we are interested in cohomology calculations, which have the advantage of being in terms of multiplicative rather than comultiplicative structure, we must dualize this model.  Let $(\Lambda v^\sharp\otimes CU^*Y, D^\sharp)$ denote the $\Bbbk$-dual of $( \Gamma v\otimes CU_*Y, D)$.  Note that the multiplication in this model satisfies
$$ \big((v^\sharp)^k\otimes \alpha\big)\cdot \big((v^\sharp)^l\otimes \beta\big)=(v^\sharp)^{k+l}\otimes \alpha\beta,$$
for all $k,l\geq 0$ and all $\alpha, \beta \in CU^*Y$.

  The dual
\begin{equation}\label{eqn:cochain-model-circle}
\xymatrix {CU^*(Y_{hS^1})\ar [r]^(0.4)\simeq &(\Lambda v^\sharp\otimes CU^*Y, D^\sharp)}
 \end{equation}
of the quasi-isomorphism in Theorem \ref{thm:moore-circle} induces an algebra map in cohomology and fits into a commutative diagram
 \begin{equation}\label{eqn:hos-model}
 \xymatrix{(\Lambda v^\sharp, 0)\ar[r]^(0.4){}&(\Lambda v^\sharp\otimes CU^*Y, D^\sharp)\ar[r]^(0.6){}&CU^*Y\\
  CU^*BS^1\ar[u]^\simeq\ar[r]&CU^* Y_{hS^1}\ar[u]^\simeq\ar[r]&CU^*Y \ar[u]^=}
 \end{equation}
  This is the \emph{cubical $S^1$-homotopy orbit model}.

A simple dualization calculation enables us to describe $D^\sharp$ completely. For each $n\geq 0$, define $\xymatrix@1{\omega _n: CU^*Y\ar[r]&CU^{*-(2n+1)}Y}$ to be the $\Bbbk$-dual of $\kappa (T_n\otimes -)$.

\begin{lem}If $\alpha\in CU^mY$, then 
$$D^\sharp\big ((v^\sharp)^n\otimes \alpha \big)=(v^\sharp) ^n \otimes d^\sharp \alpha-\sum _{k=0}^{\lceil \frac {m-2n-1}2\rceil}(v^\sharp) ^{k}\otimes\omega _{n-k-1}(\alpha)$$
where $d^\sharp$ denotes the differential of $CU^*Y$.
\end{lem}

As a consequence of this description of $D^{\sharp}$, we obtain the following useful properties of the operators $\omega _k$.

\begin{cor} The operators $\omega _n$ satisfy the following properties.
\begin{enumerate}
\item  For all $n\geq 1$, $d^\sharp\omega _n + \omega_{n}d^\sharp=\sum _{k=0}^{n-1}\omega _k\circ \omega _{n-k-1}$, while $d^\sharp \omega _0=-\omega _0d^\sharp$.
\item Each $\omega _n$ is a derivation, i.e.,  $\omega _n(\alpha \beta)= \omega _n(\alpha)\cdot \beta +\sn {\alpha}\alpha\cdot \omega _n(\beta )$.
\end{enumerate}\end{cor}

\begin{proof}  The proof of (1) proceeds by expansion of the equation $0=(D^\sharp)^2(1\otimes f)$.  To prove (2), expand the equation 
$$D^{\sharp} (1\otimes \alpha\beta)=D^{\sharp} (1\otimes \alpha)\cdot (1\otimes \beta ) +(-1)^\alpha(1\otimes \alpha)\cdot D^{\sharp}(1\otimes \beta).$$
The differential $D^{\sharp}$ is a derivation, since it is the dual of the differential of a chain coalgebra.
\end{proof} 

\begin{rmk} This corollary implies that $\omega _0$ induces a derivation of degree $-1$ 
$$\xymatrix@1{\varpi: \operatorname{H} ^*Y\ar [r]& \operatorname{H}^{*-1}Y}$$
such that $\varpi^2=0$.  Let $\op C_{2}$ denote the topological ``little squares'' operad, the homology of which is equivalent to the Gerstenhaber operad $\op G$ governing Gerstenhaber algebras.  It is well known that $\op C_{2}(2)$ is homotopy equivalent to $S^1$, so that the generator of  $\H_{1}S^1$ corresponds to the Gerstenhaber bracket operation \cite {kontsevich}.  The derivation $\varpi$ must therefore be closely related to the Gerstenhaber bracket, since a representative of the generator of  $\H_{1}S^1$ gives rise to it.  It is in fact the $\Delta$-operation of the Batalin-Vilkovisky structure on $\H^*Y$ \cite {getzler}. 
\end{rmk}

\section{Homotopy orbits of actions of simplicial groups}\label{sec:simpl-loop}

We now apply  Theorem \ref{thm:enriched-moore-simpl}  to constructing a particularly simple model for the homotopy orbits of the action of simplicial groups that are homotopy equivalent  to the loops on a simplicial suspension.  

\subsection{The canonical enriched Adams-Hilton model}

Let $\mathsf G$ denote the Kan loop group functor, which associates a simplicial group to any reduced simplicial set \cite{may}.  Recall that for any reduced simplicial set $K$, the geometric realization of $\mathsf G K$ is homotopy equivalent to the based loop space on the realization of $K$.

Szczarba proved long ago in \cite{szczarba} that for any reduced simplicial set $K$, there is a natural quasi-isomorphism of chain algebras
$$Sz_{K}:\Om C_{*}K \xrightarrow\simeq C_{*}\mathsf G K,$$
(cf.~Example \ref{ex:szczarba}) so that $\Om C_{*}K$ provides a good model for the multiplicative structure in the chain Hopf algebra $C_{*}\mathsf G K$.  It is natural to ask to what extent $\Om C_{*}K$ also captures the comultiplicative structure of $C_{*}\mathsf G K$.  
 
 We recall here the results leading up to the conclusion in \cite{hpst} that $\Om C_{*}K$ admits a natural comultiplication $\psi_{K}$ with respect to which $Sz_{K}$ is a DCSH-multiplicative map.  The Hopf algebra $(\Om C_{*}K, \psi_{K})$ thus captures both the multiplicative and the comultiplicative structure of $C_{*}\mathsf G K$. 

\begin{thm}\cite {gugenheim-munkholm}\label{thm:gm-dcsh} Let $K$ be a reduced simplicial set. The natural comultiplication $\delta_{K}:C_*K\to C_*K\otimes C_*K$ is naturally a DCSH-map, i.e., there is a chain algebra map
$$\varphi_{K}:\Om C_*K\to \Om \big(C_*K\otimes C_*K\big),$$
natural in $K$, such that $(\varphi_{K})_{1}=\delta _{K}$. 
\end{thm}

\begin{thm}\cite{hpst}\label{thm:aw-diagonal}The composite chain algebra map 
$$\Om C_{*}K\xrightarrow {\varphi_{K}}\Om (C_{*}K\otimes C_{*}K)\xrightarrow q \Om C_{*}K\otimes \Om C_{*}K,$$
denoted $\psi_{K}$, endows $\Om C_{*}K$ with a natural chain Hopf algebra structure.
\end{thm}

The comultiplication $\psi_{K}$ is called the \emph{Alexander-Whitney diagonal} or the \emph{canonical cobar diagonal}.

\begin{thm}\cite{hpst} The Szczarba quasi-isomorphism of chain algebras 
$$Sz_{K}:\Om C_{*}K\to C_{*}\mathsf GK$$  
is a multiplicative DCSH map, with respect to the Alexander-Whitney diagonal on $\Om C_{*}K$ and the usual comultiplication on $C_{*}\mathsf GK$.
\end{thm}

For the construction of our coalgebraic model of the homotopy orbits of a $\mathsf GK$-action, we need to know that the natural coalgebra structure on $\Om C_{*}K$ extends to a coalgebra structure on the acyclic cobar construction $C_{*}K\otimes _{t_{\Om}} \Om C_{*}K$, which was proved in \cite{hps3}, at least for simplicial suspension. 

Recall that if $\mathsf E$ denotes the simplicial suspension functor \cite{may}, and $K=\mathsf EK'$ for some simplicial set $K'$, then  the generators of the free abelian group $C_{n+1}K$ are in natural, bijective correspondence with the generators of $C_{n}K'$, for all $n\geq 0$. If $x$ is a generator of $C_{n}K'$, let $e(x)$ denote the corresponding generator of $C_{n+1}K$ . 

\begin{thm} \cite{hps3} If $K=\mathsf EK'$ for some simplicial set $K'$, then there is a $\Om C_{*}K$-semifree extension of $\Om C_{*} K$-module coalgebras
$$\Om C_{*} K\to C_{*}K\otimes _{t_{\Om}} \Om C_{*}K,$$
where the comultiplication $\widehat \psi_{K}$ on $C_*K\otimes _{t_{\Om}} \Om C_{*}K$  
satisfies
\begin{align*}
\widehat \psi_{K}\big (e(x)\otimes w)=&\big(e(x)\otimes w_{j}\big)\otimes \big( 1 \otimes w^{j}\big)\\
&+\big(1\otimes w_{j}\big)\otimes \big( e(x) \otimes w_{j}\big)\\
&\pm \big(1\otimes \si e(x_{i})\cdot w_{j}\big)\otimes \big( e(x^{i}) \otimes w^{j}\big),
\end{align*}
where  $\delta_{K'}(x)=x_{i}\otimes x^{i}$ and $\psi_{K}(w)=w_{j}\otimes w^j$.
\end{thm}

\subsection {Modeling $\mathsf GK$-homotopy orbits}

Let $K$ be the simplicial suspension of a simplicial set $K'$. Let $E$ be a contractible simplicial set that admits a free right $\mathsf GK$-action and that is the total space of a principal twisted cartesian product with fiber $\mathsf GK$.  For example, the construction $\mathsf W\mathsf GK$ of  \cite{may} is one possible choice of $E$.  Let $j:\mathsf GK\hookrightarrow E$ denote the inclusion of the fiber.

If $L$ is a simplicial set admitting a left $\mathsf GK$-action, then a model of the simplicial set of homotopy orbits of the $\mathsf GK$ action on $L$ is
$$L_{h \mathsf GK} :=E\times_{\mathsf GK} L.$$  
We construct here a small, simple and totally explicit chain coalgebra model for $L_{h \mathsf GK}$.

 Consider the following commutative diagram of right $\Om C_{*}K$-module coalgebras.
\begin{equation}\label{eqn:diag-gk}
\xymatrix{\Om C_{*}K\ar [d]^{\iota}\ar [r]^{C_{*}j\circ Sz_{K}}& C_{*}E\ar @{->>}[d]_\simeq\\
C_{*}K\otimes _{t_{\Om}}\Om C_{*}K\ar [r]^(.7){\simeq}&\mathbb  Z}
\end{equation}
The inclusion $\iota$ is a coalgebra map, and the other vertical arrow is a surjective quasi-isomorphism, while the two horizontal maps are DCSH-module maps with respect to $Sz_{K}$.  We can therefore apply Theorem \ref{thm:exist-dcsh-res} to diagram (\ref{eqn:diag-gk}) and obtain  a DCSH $\Om C_{*}K$-resolution of $C_{*} E$:
\begin{equation}\label{eqn:dcsh-resoln-gk}
C_{*}K\otimes _{t_{\Om}}\Om C_{*}K\xrightarrow{\simeq}C_{*}E.
\end{equation}

Theorem \ref{thm:enriched-moore-simpl} applied to the DCSH-resolution (\ref{eqn:dcsh-resoln-gk}) implies the existence of a chain coalgebra model for the homotopy orbits of a left $\mathsf GK$-action, as stated precisely below, where we use that 
$$(C_{*}K\otimes _{t_{\Om}}\Om C_{*}K)\ot_{\Om C_{*}K}C_{*}L\cong C_{*}K\ot_{Sz_{K}\circ t_{\Om}}C_{*}L.$$

\begin{thm}\label{thm:new-moore-gk}  Let $K$ be a simplicial suspension, and let $L$ be a simplicial set admitting a left  $\mathsf G K$-action. There exist
\begin{enumerate}
\item a coassociative comultiplication on the twisted tensor product $C_{*}K\otimes _{t_{\Om}}C_*L$, extending the comultiplication on $C_{*}L$; and
\item a DCSH map $C_{*}K\otimes _{Sz_{K}\circ t_{\Om}}C_*L\to C_{*}L_{h\mathsf GK}$ that is a quasi-isomorphism.
\end{enumerate}
\end{thm}

\begin{rmk}  Let $D$ denote the differential in $C_{*}K\otimes _{Sz_{K}\circ t_{\Om}}C_*L$.  Since the comultiplication in $C_{*}K$ is trivial, the formulas in Definition \ref{defn:ttp} imply that for all $x\in C_{m}K$ and $y\in C_{n}L$,
$$D(x\otimes y)= dx \otimes y +(-1)^m x\otimes dy -1\otimes sz_{K}(x)\cdot y,$$
where $sz_{K}:C_{*}K\to \Om C_{*}K$ is the twisting cochain of Example \ref{ex:szczarba}.
Moreover, the comultiplication $\widetilde \psi$ on $C_{*}K\otimes _{Sz_{K}\circ t_{\Om}}C_*L$, which is induced by $\widehat \psi_{K}$ and the usual comultiplication $\delta_{L}$ on $C_{*}L$, satisfies
\begin{align*}
\widetilde\psi (x\otimes y)=&\big(e(x)\otimes y_{j}\big)\otimes \big( 1 \otimes y^j\big)\\
&+\big(1\otimes y_{j}\big)\otimes \big( e(x) \otimes y^j\big)\\
&\pm \big(1\otimes sz_{K}\big( e(x_{i})\big)\cdot y_{j}\big)\otimes \big( e(x^{i}) \otimes y^j\big),
\end{align*}
\end{rmk}

\begin{rmk} If $K$ and $L$ both have only a finite number of nondegenerate simplices, then  the model of Theorem \ref{thm:new-moore-gk} for $L_{h\mathsf GK}$ is quite small and should lend itself easily to explicit computation of comultiplicative structure in $\H_{*}L_{h\mathsf GK}$, or, dually, of multiplicative structure in $\H^{*}L_{h\mathsf GK}$.
\end{rmk}

\appendix

\section{Twisting cochains}\label{sec:twistingcochains}

We begin by recalling the cobar and bar constructions in the differential graded framework.  
Let $\cat{Coalg}_{\Bbbk}$ denote the category of $1$-connected, coaugmented chain coalgebras over a commutative ring $\Bbbk$, i.e., of coaugmented comonoids in $\cat {Ch}_{\Bbbk}$ such that $C_{<0}=0$, $C_{0}=\Bbbk$, and $C_{1}=0$. Let $\cat {Alg}_{\Bbbk}$ denote the category of connected, augmented chain algebras over $\Bbbk$, i.e., of augmented monoids $B$ in $\cat {Ch}_{\Bbbk}$ such that $B_{<0}=0$ and $B_{0}=\Bbbk$. 

 The \emph{cobar construction} functor $\Om:\cat {Coalg}_{\Bbbk}\to \cat {Ch}_{\Bbbk}$, defined by 
$$\Om C= \left(T (\si C_{+}), d_{\Om}\right)$$
where, if $d$ denotes the differential on $C$, then
\begin{align*}
d_{\Om}(\si c_{1}|\cdots|\si c_{n})=&\sum _{1\leq j\leq n}\pm \si c_{1}|\cdots |\si (dc_{j})|\cdots |\si c_{n}\\ 
&+\sum _{1\leq j\leq n}\pm \si c_{1}|...|\si c_{ji}|\si c_{j}{}^{i}|\cdots |\si c_{n},
\end{align*}
with signs determined by the Koszul rule, where the reduced comultiplication applied to $c_{j}$ is $c_{ji}\otimes c_{j}{}^{i}$ (using Einstein implicit summation notation).  

The graded $\Bbbk$-module underlying $\Om C$ is naturally a free associative algebra, with multiplication given by concatenation. The differential $d_{\Om }$ is a derivation with respect to this concatenation product, so that $\Om C$ is itself a chain algebra.  We therefore consider the cobar construction to be a functor 
$$\Om :  \cat {Coalg}_{\Bbbk}\to \cat {Alg}_{\Bbbk}.$$  

 The \emph{bar construction} functor from $\cat{Alg}_{\Bbbk}$ to $\cat {Ch}_{\Bbbk}$, defined by 
$$\Bar B=\left(T (sB_{+}), d_{\Bar}\right)$$
where, if $d$ is the differential on $B$, then (modulo signs, which are given by the Koszul rule)
\begin{align*}
d_{\Bar}(sb_{1}|\cdots|sb_{n})=&\sum _{1\leq j\leq n}\pm sb_{1}|\cdots |s(db_{j})|\cdots |sb_{n}\\ 
&+\sum _{1\leq j<n}\pm sb_{1}|...|s(b_{j}b_{j+1})|\cdots |sb_{n}.
\end{align*}

The graded $\Bbbk$-module underlying $\Bar B$ is naturally a cofree coassociative coalgebra, with comultiplication given by splitting of words. The differential $d_{\Bar}$ is a coderivation with respect to this splitting comultiplication, so that $\Bar B$ is itself a chain coalgebra. We  therefore consider the bar construction to be a functor 
$$\Bar :  \cat {Alg}_{\Bbbk}\to \cat {Coalg}_{\Bbbk}.$$

 Let $\eta: Id\to \Bar \Om$ denote the unit of the cobar/bar adjunction.  It is well known that for all $1$-connected, coaugmented chain coalgebras $C$, the counit map
\begin{equation}\label{eqn:unit-barcobar}
\eta_{C}:C\xrightarrow\simeq\Bar\Om C
\end{equation}
is a quasi-isomorphism of chain coalgebras \cite[Corollary 10.5.4]{neisendorfer}.  

\begin{defn}
A \emph{twisting cochain} from a $1$-connected, coaugmented chain coalgebra $(C,d)$ with comultiplication $\Delta$ to a connected, augmented chain algebra $(A,d)$ with multiplication $m$ consists of a linear map $t:C\to A$ of degree $-1$ such that
$$dt+td=m (t\otimes t)\Delta.$$
\end{defn}

\begin{rmk} If $t:C\to A$ is a twisting cochain, then $ftg:C'\to A'$ is also a twisting cochain, for every coalgebra morphism $g:C'\to C$ and every algebra morphism $f:A\to A'$.
\end{rmk}

\begin{rmk}
A twisting cochain $t:C\to A$ induces both a chain algebra map
$$\alpha _{t}:\Om C\to A$$
specified by $\alpha _{t}(\si c)=t(c)$ and a chain coalgebra map (the adjoint of $\alpha_{t}$ under the $(\Om, \Bar)$-adjunction)
$$\beta _{t}:C\to \Bar A, $$
 satisfying
 $$\alpha_{t}=\ve_{A}\circ \Om \beta_{t}\quad\text{and}\quad \beta _{t}=\Bar\alpha_{t}\circ \eta_{C}.$$
It follows that $\alpha_{t}$ is a quasi-isomorphism if and only if $\beta_{t}$ is a quasi-isomorphism.  
\end{rmk}
 
 \begin{ex} Let $C$ be a $1$-connected, coaugmented chain coalgebra. The \emph{universal  twisting cochain}
$$t_{\Om}:C\to \Om C$$
is defined by $t_{\Om }(c)=s^{-1} c$ for all $c\in C$, where $\si c$ is defined to be $0$ if $|c|=0$.  Note that  $\alpha_{t_{\Om}}=Id_{\Om C}$, so that $\beta _{t_{\Om}}=\eta_{C}$.  Moreover, $t_{\Om}$ truly is universal, as all twisting cochains $t:C\to A$ factor through $t_{\Om}$, since the diagram
 $$\xymatrix{C\ar[r] ^{t_{\Om}}\ar [dr]_{t}&\Om C\ar [d]^{\alpha_{t}}\\ &A}$$
 always commutes.
 \end{ex}
 
  \begin{ex}\label{ex:couniversal} Let $A$ be a connected, augmented chain algebra. The \emph{couniversal  twisting cochain}
$$t_{\Bar}:\Bar A\to A$$
is defined by $t_{\Bar  }(sa)=a$ for all $a\in A$ and $t_{\Bar}(sa_{1}|\cdots |sa_{n})=0$ for all $n>1$.  Note that  $\beta_{t_{\Bar}}=Id_{\Bar A}$, so that $\alpha _{t_{\Om}}=\ve_{A}$.  Moreover, $t_{\Bar}$ truly is couniversal, as all twisting cochains $t:C\to A$ factor through $t_{\Bar}$, since the diagram
 $$\xymatrix{&\Bar A\ar [d]^{t_{\Bar}}\\ C\ar[ur]^{\beta_{t}}\ar [r]_{t}&A}$$
 always commutes.
 \end{ex}

 \begin{ex} \label{ex:szczarba} Let $K$ be a reduced simplicial set, and let $\G K$ denote its Kan loop group. In 1961 \cite {szczarba}, Szczarba gave an explicit formula for a twisting cochain
$$sz_{K}:C_{*}K\to C_{*}\G K,$$
natural in $K$, and proved that 
\begin{equation}\label{eqn:szczarba}
Sz _{K}:=\alpha_{sz_{K}}:\Om C_{*}K\to C_{*}\G K
\end{equation}
was a quasi-isomorphism of chain algebras for every $K$.  It follows that the induced coalgebra map
$$Sz_{K}^\sharp:=\beta_{sz_{K}}:C_{*}K\to \Bar C_{*}\G K$$
is also a quasi-isomorphism.
\end{ex}

\begin {defn}\label{defn:ttp} Let $t:C\to A$ be a twisting cochain.  Let $M$ be a right $A$-module, 
where $\rho :M\otimes A\to M$ is the $A$-action, and let  $N$ be a left $C$-comodule, where $\lambda 
:N\to C\otimes N$ is the $C$-coaction.  Let $d$ denote the differential on both $M$ and $N$. The \emph{twisted tensor product }ﾊof $M$ and $N$ over $t$ is a chain complex  $M\ot_tN=(M\otimes N, 
D_{t})$, where $$D_{t}=d\otimes 1+1\otimes d+
(\rho\otimes 1)(1\otimes t\otimes 1)(1\otimes 
\lambda).$$

If, on the other hand,  $M$ is a left $A$-module, 
with $A$-action $\lambda :A\otimes M\to M$, and  $N$ is a right $C$-comodule, with $C$-coaction $\rho 
:N\to  N\otimes C$, then the \emph{twisted tensor product }ﾊof $M$ and $N$ over $t$ is a chain complex  $N\ot_t M=(N\otimes M, 
D_{t})$
$$D_t = d \otimes 1 + 1\otimes d - (1 \otimes \lambda)(1\otimes t\otimes 1)(\rho\otimes 1).$$
\end{defn}

\begin{ex}\label{ex:acyc-bar}  For any 1-connected, coaugmented chain coalgebra $C$, the twisted tensor products
$$C\ot_{t_{\Om}}\Om C\quad\text{and}\quad \Om C\otimes_{t_{\Om}} C$$
are the usual acyclic cobar constructions.  Similarly, for any connected, augmented chain algebra $A$, the twisted tensor products
$$A\ot_{t_{\Bar}}\Bar A\quad\text{and}\quad \Bar A\otimes_{t_{\Bar}} A$$
are the usual acyclic bar constructions.
\end{ex}

 \bibliographystyle{amsplain}
\bibliography{htpyorb}
\end{document}